\documentclass[12pt,reqno]{amsart}
\usepackage{amsfonts}
\usepackage{amssymb,color}
\usepackage{amssymb}

\def\half#1{\frac{#1}{2}}

\def\norm#1{\|#1\|}
\def\normo#1{\left\|#1\right\|}
\newcommand{\les}{{\lesssim}}

\newcommand{\T}{\mathbb{T}}

\newcommand{\N}{\mathbb{N}}
\newcommand{\R}{\mathbb{R}}
\newcommand{\Z}{\mathbb{Z}}

\newcommand{\gau}{\mathcal{G}}

\newcommand{\ft}{{\mathcal{F}}}
\newcommand{\Hl}{{\mathcal{H}}}
\newcommand{\sgn}{{\mbox{sgn}}}

\newcommand{\Del}[1]{}

\numberwithin{equation}{section}

\newtheorem{thm}{Theorem}[section]

\theoremstyle{remark}

\newtheorem{theorem}{Theorem}[section]
  \newtheorem{proposition}[theorem]{Proposition}
  \newtheorem{lemma}[theorem]{Lemma}

\theoremstyle{remark}

\theoremstyle{definition}

\begin{document}

\title[Periodic mBO equation]{Well-posedness in energy space for the periodic modified Benjamin-Ono equation}
\author[Z. Guo, Y. Lin, L. Molinet]{Zihua Guo$^{1,2}$, Yiquan Lin$^{1,2}$, Luc Molinet$^3$}

\address{$^1$LMAM, School of Mathematical Sciences, Peking
University, Beijing 100871, China}
\address{$^2$Beijing International Center for Mathematical Research, Beijing
100871, China} \email{zihuaguo@math.pku.edu.cn,
yqlin.math@gmail.com}

\address{$^3$Universit\'{e} Francois Rabelais Tours, F\'{e}d\'{e}ration Denis Poisson-CNRS, Parc
Grand-mont, 37200 Tours, France}

\email{luc.molinet@lmpt.univ-tours.fr}

\begin{abstract}
We prove that the periodic modified Benjamin-Ono equation is
locally well-posed in the energy space $H^{1/2}$. This ensures the global well-posedness in the defocusing case.
The proof is
based on an  $X^{s,b}$ analysis of the
system after gauge transform.
\end{abstract}

\maketitle

\tableofcontents

\section{Introduction, main results and notations}

In this paper, we study the Cauchy problem for the
modified Benjamin-Ono equation on the torus that reads\begin{align}\label{eq:mBO}
\begin{cases}
\partial_t u+\Hl \partial_x^2 u=\mp u^2u_x,\\
u(x,0)=u_0
\end{cases}
\end{align}
where $u(t,x):\R\times\T \to \R$, $\T=\R/{2\pi \Z}$ and $\Hl$ is the
Hilbert transform
\[\widehat{\Hl f}(0)=0, \quad \widehat{\Hl f}(k)=-i\sgn(k)\hat{f}(k),\ k\in \Z^*.\]
This equation is called defocusing when there is a minus sign in
front of the nonlinear term $ u^2 u_x $ and focusing when it is a
plus sign. The Benjamin-Ono equation with the quadratic nonlinear
term
\begin{align}\label{eq:BO}
\partial_t u+\Hl \partial_x^2 u=&uu_x
\end{align}
was derived by Benjamin \cite{Ben} and Ono \cite{Ono} as a model for
one-dimensional waves in deep water. On the other hand, the cubic
nonlinearity is also of much interest for long wave models
\cite{ABFS,KPV}.

There are at least the three    following quantities preserved
under the flow of the real-valued mBO equation \eqref{eq:mBO}\footnote{In \eqref{eq:H1half} the $+ $ corresponds to the defocusing case whereas the $-$
 corresponds the the focusing one.}
\begin{eqnarray}\label{eq:conservation}
\int_\T u(t,x)dx& = & \int_\T u_0(x)dx, \\
\int_\T u(t,x)^2dx& =& \int_\T u_0(x)^2dx,\label{eq:L2con}\\
\int_\T \frac{1}{2}u\Hl u_x\mp \frac{1}{12}
|u(t,x)|^4dx& =& \int_\T \frac{1}{2}u_0\Hl u_{0,x}\mp \frac{1}{12}
|u_0(x)|^4dx\; .\label{eq:H1half}
\end{eqnarray}
These conservation laws provide a priori bounds on the solution. For
instance, in the defocusing case we get from \eqref{eq:L2con} and
\eqref{eq:H1half} that the $H^{1/2}$ norm of the solution remains
bounded for all  times if the initial data belongs to $H^{1/2}$.
This is crucial in order to prove the well-posedness result. On the
other hand the mBO equation is $ L^2 $-critical (in the sense that
the $ L^2(\R) $-norm  is preserved by the dilation symmetry of the
equation).  Therefore, in the focusing case, one expects that a
phenomenon  of blow-up in the energy space occurs\footnote{Progress
in this direction can be found in \cite{KMR} for the case on the
real line.}.

The Cauchy problems for \eqref{eq:mBO} and the Benjamin-Ono equation
\eqref{eq:BO} have been extensively studied. For instance, in both
real-line and periodic case, the energy method provides local
well-posedness for BO and mBO in $H^s$ for $s>3/2$ \cite{Iorio}. In
the real-line case, this result was improved by combination of
energy method and the dispersive effects. For real-line BO equation,
the result $s\geq 3/2$ by Ponce \cite{Ponce} was the first place of
such combination as a consequence of the commutator estimates in
\cite{KatoPonce}, was later improved to $s>5/4$ in \cite{KT}, and
$s>9/8$ in \cite{KK}. Tao \cite{Tao} obtained global well-posedness
in $H^s$ for $s\geq 1$ by using a gauge transformation as for the
derivative Schr\"odinger equation and Strichartz estimates. This
result was improved to $s\geq 0$ by Ionescu and Kenig \cite{IK}, and
to $s\geq 1/4$ (local well-posedness) by Burq and Planchon
\cite{BP}. Their proof both used the Fourier restriction norm
introduced in \cite{Bour}. Recently, Molinet and Pilod \cite{MP}
gave a simplified proof for $s\geq 0$ and obtained unconditional
uniqueness for $s>1/4$.

For the real-line mBO, this was improved to $s\geq 1$ by
Kenig-Koenig \cite{KK} by the enhanced energy methods. Molinet and
Ribaud \cite{MR2} obtained analytic local well-posedness for the
complex-valued mBO in $H^s$ for $s>1/2$ and $B_{2,1}^{1/2}$ with a
small $L^2$ norm, improving the result of Kenig-Ponce-Vega
\cite{KPV2} for $s>1$. The smallness condition of $H^s (s>1/2)$
results was later removed in \cite{MR} by using Tao's gauge
transformation \cite{Tao}. The result for $s=1/2$ was obtained by
Kenig and Takaoka \cite{KenigT} by using frequency dyadically
localized gauge transformation. Their result is sharp in the sense
that the solution map is not locally uniformly continuous in $H^s$
for $s<1/2$ (The failure of $C^3$ smoothness was obtained in
\cite{MR2}). Later, Guo \cite{Guo} obtained the same result without
using gauge transform under a smallness condition on the $L^2$ norm.

In the periodic case, there is no smoothing effect for the equation.
However, to overcome the loss of derivative, the gauge transform
still applies. For the periodic BO equation, global well-posedness
in $H^1$ was proved by Molinet and Ribaud \cite{MRmBO}, was later
improved by Molinet to $H^{1/2}$ \cite{Mol2}, and $L^2$ \cite{Mol1}.
Molinet \cite{Mol3} also proved that the result in $L^2$ is sharp in
the sense that the solution map fails to be continuous below $L^2$.
For the periodic mBO \eqref{eq:mBO}, local well-posedness in $H^1$
was proved in \cite{MRmBO}. Their proof used the Strichartz norm and
gauge transform.

The purpose of this paper is to improve the well-posedness results
for \eqref{eq:mBO} to the energy space $H^{1/2}$ and, as a by-product, to prove that the solutions can be extended for all times in the defocusing case.
 The main result of
this paper is
\begin{thm}\label{thm}
Let $s\geq 1/2$. For any intial data $ u_0\in H^s(\T) $  there exists $ T= T(\|u_0\|_{H^{1/2}} )>0$ such that the mBO equation \eqref{eq:mBO}
 admits a unique solution $$u\in C([-T,T];H^s(\T) )\quad \mbox{Êwith }Ê\quad  P_+(e^{iF(u)}) \in X^{\frac{1}{2},\frac{1}{2}}_T \; . $$
Moreover, the solution-map $ u_0 \mapsto u $ is continuous from the ball of $ H^{1/2}(\T) $ of radius $\|u_0\|_{H^{1/2}}$, equipped with the $ H^s(\T) $-topology,  with values in  $ C([-T,T];H^s(\T)) $. \\
Finally, in the defocusing case, the solution can be extended for all times and belongs to $ C(\R;H^s(\T))\cap C_b(\R;H^{1/2}(\T))$.
\end{thm}

A very similar equation to mBO \eqref{eq:mBO} is the derivative
nonlinear Schr\"odinger equation
\begin{align}\label{eq:dNLS}
\begin{cases}
i\partial_t u+\partial_x^2 u=i(|u|^2u)_x,\quad (t,x)\in \R\times \T\\
u(0,x)=u_0.
\end{cases}
\end{align}
It has also attracted extensive attention. Local well-posedness for
\eqref{eq:dNLS} in $H^{1/2}$ was proved by Herr \cite{Herr}. There
are several differences between \eqref{eq:mBO} and \eqref{eq:dNLS}.
The first one is the integrability: \eqref{eq:dNLS} is integrable
while \eqref{eq:mBO} is not. The second one is the conservation
laws: \eqref{eq:mBO} has a conservation law at level $H^{1/2}$, and
hence GWP in $H^{1/2}$ is much easier. The last one is the action of
the gauge transform: let $v$ be the function after gauge transform,
\eqref{eq:dNLS} can be reduced to a clean equation which involves
only $v$, while \eqref{eq:mBO} can only reduce to a system that involves
both $u$ and $v$, and hence the gauge for \eqref{eq:mBO} brings more
technical difficulties.

We discuss now the ingredients in the proof of Theorem \ref{thm}.
Let $u$ be a smooth
solution to \eqref{eq:mBO}, define
\begin{align}
w=T(u):=\frac{1}{\sqrt{2}}u(t,x-\int_0^t \frac{1}{2\pi}\int
u^2(s,x)dxds).
\end{align}
Then $w$ solves the "Wicked order" mBO equation:
\begin{align}\label{eq:mBO2}
\begin{cases}
\partial_t u+\Hl \partial_x^2 u=2P_{\ne c}(u^2)u_x,\\
u(x,0)=u_0,
\end{cases}
\end{align}
where $P_{\ne c}f=f-\frac{1}{2\pi}\int_{\T} fdx$. It is easy to see
that $T$ and its inverse $T^{-1}$ are both continuous maps from
$C((-T,T):H^s)$ to $C((-T,T):H^s)$ for $s\geq 0$. Therefore we will  consider \eqref{eq:mBO2} instead of  \eqref{eq:mBO}.
Now, in
order to overcome the loss of derivative, we will apply the method
of gauge transform as in \cite{MRmBO, Mol1, Mol2}, which was first
developed for BO equation by Tao \cite{Tao}. As noticed above the equation satisfied by this  gauge transform $ v$
 involves terms with both $ u $ and $ v $.
 One of the main difficulties is that the solution $ u$ does not share the same regularity in
 Bourgain'space as the gauge transform $v$.
The main new ingredient is the  use of the Marcinkiewicz
 multiplier theorem that enables us to treat the multiplication by $u $  in Bourgain'space in a simple way.
\subsection{Notations}
For $A,B>0$, $A\lesssim B$ means that there exists $c>0$ such that
$A\leq cB$. When $c$ is a small constant we use $A\ll B$. We write
$A\sim B$ to denote the statement that $A\lesssim B\lesssim A$.

We denote the sum on $\Z$ by integral form $\int
a(\xi)d\xi:=\sum_{\xi\in\Z}a(\xi)$. For a $2\pi$-periodic function
$\phi$, we define its Fourier transform on $\Z$ by
\begin{align*}
\hat{\phi}(\xi):=\int_{\R/2\pi\Z}e^{-i\xi x}\phi(x)dx,\ \forall\
\xi\in \Z.
\end{align*}
We denote by $W(\cdot)$ the unitary group
$W(t)u_0:=\ft_x^{-1}e^{-it|\xi|\xi}\mathcal{F}_xu_0(\xi)$.

For a function $u(t,x)$ on $\R\times\R/(2\pi)\Z$, we define its
space-time Fourier transform as follows, $\forall\ (\tau,\xi)\in
\R\times \Z$
\begin{align*}
\hat{u}(\tau,\xi):=\mathcal{F}_{t,x}(u)(\tau,\xi):=\mathcal{F}(u)(\tau,\xi)
=\int_{\R}\int_{\R/(2\pi)\Z}e^{-i(\tau t+\xi x)}u(t,x)dxdt.
\end{align*}
Then define the Sobolev spaces $H^s$ for $(2\pi)$-periodic function
by
\begin{align*}
\|\phi\|_{H^s}:=\|\langle\xi\rangle^s\hat{\phi}\|_{l^2_{\xi}}
=\|J_x^s\phi(x)\|_{L^2_x},
\end{align*}
where $\langle\xi\rangle:=(1+|\xi|^2)^{\frac{1}{2}}$ and
$\widehat{J^s_x\phi}(\xi) :=\langle\xi\rangle^s\hat{\phi}(\xi)$. For
$ 2<q<\infty $ we define also the Sobolev type spaces $ H^s_q $ by
$$
\|\phi\|_{H^s_q} :=\| J_s^x \phi \|_{L^q} \; .
$$
We will use the following Bourgain-type spaces denoted by $X^{s,b},\
Z^{s,b}$ and $Y^{s}$ of $(2\pi)$-periodic (in $x$) functions
respectively endowed with the norm
\begin{align*}
\|u\|_{X^{s,b}}:=&\|\langle\xi\rangle^s
\langle\tau+|\xi|\xi\rangle^b\hat{u}(\tau,\xi)\|_{L^2_{\tau,\xi}},
\\
\|u\|_{Z^{s,b}}:=&\|\langle\xi\rangle^s
\langle\tau+|\xi|\xi\rangle^b\hat{u}(\tau,\xi)\|_{L^2_{\xi}L^1_{\tau}},
\end{align*}
and
\begin{align}
\|u\|_{Y^{s}}:=\|u\|_{X^{s,\frac{1}{2}}}+\|u\|_{Z^{s,0}}.
\end{align}
One can easily check that $ u \mapsto \overline{u}$ an isometry in  $ X^{s,b} $ and $ Z^{s,b} $ and that  $Y^{s}\hookrightarrow Z^{s,0}\hookrightarrow
C(\R;H^s)$. We will also use the space-time Lebesgue spaces denoted
by $L^p_tL^q_x$  of $(2\pi)$-periodic (in
$x$) functions endowed with the norm
\begin{align*}
\|u\|_{L^p_tL^q_x}:=\Big(\int_{\R}\|u(t,\cdot)\|_{L^q_x}^pdt\Big)^{1/p},\,
\end{align*}
with the obvious modification for $p=\infty$. For any space-time
function space $B$ and any $T>0$, we denote by $B_{T}$ the
corresponding restriction in time space endowed with the norm
\begin{align*}
\|u\|_{B_T}:=\inf_{v\in B}\{\|v\|_{B},v(\cdot)\equiv u(\cdot) \ on \
(0,T)\}.
\end{align*}

Let $\eta_0: \R\rightarrow [0, 1]$ denote an even smooth function
supported in $[-8/5, 8/5]$ and equal to $1$ in $[-5/4, 5/4]$. For
$k\in \N^*$ let $\chi_k(\xi)=\eta_0(\xi/2^{k-1})-\eta_0(\xi/2^{k-2})$,
$\eta_{\leq k}=\eta_0(\xi/2^{k-1})$, and then let $  P_{2^k}$ and $ P_{\leq 2^{k}}$
denote the operators on $L^2(\T)$ defined by
\[  \widehat{P_{1}u}(\xi)=\eta_0(2\xi) , \; \widehat{P_{2^k}u}(\xi)=\chi_k(\xi)\widehat{u}(\xi) \, , k\in \N^*,   \mbox{ and } \widehat{P_{\leq 2^k}u}(\xi)=\eta_{\leq k}(\xi)
 u(\xi)
 \; . \] By a slight abuse
of notation we  define the operators $P_{2^k},P_{\leq {2^k}}$ on
$L^2(\R\times \T)$ by the  formulas
$\ft(P_{2^k} u)(\tau,\xi)=\chi_k(\xi)\ft(u)(\tau,\xi)$,
$\ft(P_{\le 2^k}u)(\tau,\xi)=\eta_{\leq k}(\xi)\ft(u)(\tau,\xi)$. We also define
 the projection operators $P_\pm f=\ft^{-1}1_{\pm\xi>0}\ft f$,
$P_cf=\frac{1}{2\pi}\int_{\T} fdx$, $P_{\ne c}=I-P_c$, and
$P_{2^k}^+=P_+P_{2^k}$, $P_{\leq 2^k}^+=P_+P_{\leq 2^k}$.

To simplify the notation, we  use capitalized variables to
describes the dyadic number, i.e. any capitalized variables such as
$N $ range over the dyadic
 number $ 2^{\N} $. \\
Finally, for any $ 1\le p\le \infty $ and any function space  $ B $ we define the space-time function space $ \widetilde{L^p_t B} $ by
$$
\|u\|_{\widetilde{L_t^p B}} := \Bigl( \sum_{k=0}^{\infty} \| P_{2^k} u \|_{L^p_t B}^2 \Bigr)^{1\over 2}\; .
$$
It is worth noticing that Littlewood-Paley square function theorem
 ensures that $ \widetilde{L^p_tL^p_x}\hookrightarrow L^p_t L^p_x $ for $ 2\leq p<\infty $.

\section{Gauge transform}

In this section, we introduce the gauge transform. Let $u\in
C([-T,T]:H^\infty(\T))$ be a smooth solution to \eqref{eq:mBO2}.
Define the periodic primitive of $u^2-\frac{1}{2\pi}\norm{u(t)}_2^2$
with zero mean by
\[F=F(u)=\partial_x^{-1}P_{\ne c}(u^2)=\frac{1}{2\pi}\int_0^{2\pi}\int_\theta^x u^2(t,y)-\frac{1}{2\pi}\norm{u(t)}_{L^2}^2dyd\theta.\]
Let
\begin{align}\label{eq:gauge}
v=\gau(u):=P_+(e^{-iF}u),
\end{align}
then we look for the equation that $v$ solves. It holds
\begin{align*}
v_t=&P_+[e^{-iF}(-iF_tu+u_t)],\\
v_{xx}=&P_+[e^{-iF}(-F_x^2u-iF_xu_x-i(F_xu)_x+u_{xx})] \; ,
\end{align*}
and thus
\begin{align*}
v_t-iv_{xx}
=&P_+[e^{-iF}(-iF_tu+i(F_x)^2u-F_{xx} u)]\\
&+P_+[e^{-iF}(u_t-iu_{xx}-2 F_x u_x)]:=I+II.
\end{align*}
Using equation \eqref{eq:mBO2} we easily get
\begin{align*}
II=&P_+[e^{-iF}(u_t+\Hl u_{xx}-2iP_{-} u_{xx}-2 F_x u_x)]=-2i P_+[e^{-iF} P_{-} u_{xx}].
\end{align*}
Next we compute $I$. Using again  \eqref{eq:mBO2} and the conservation of the $ L^2 $-norm for smooth solutions, we have
\begin{align}
 F_t=&\partial_t \partial_x^{-1} (P_{\neq c} u^2) =  \partial_x^{-1} \partial_t   ( u^2-P_{c}u^2) =   \partial_x^{-1} \partial_t   u^2 \nonumber \\
= & 2\partial_x^{-1} \Bigl( -u\Hl u_{xx} +2 P_{\neq c}?u^2) u u_x  \Bigr)   \nonumber\\
= &  2\partial_x^{-1} \Bigl(-\partial_x( u \Hl u_x) + u_x \Hl u_x + P_{\neq c} (u^2) \partial_x P_{\neq c}(u^2) \Bigr) \nonumber \\
= & P_{\neq c}\Bigl( (P_{\neq c}u^2)^2\Bigr) -2 u \Hl u_x+2 P_{c}?u \Hl u_x) +2 \partial_x^{-1} (u_x \Hl u_x)\; . \label{F_t}
\end{align}
Noticing that $ F_x= P_{\neq c} (u^2) $ we infer that
$$
 -iuP_{\neq c}\Bigr(  (P_{\neq c} u^2)^2\Bigr)+iu (F_x)^2  =iu P_{c} \Bigl( (P_{\neq c} u^2)^2 \Bigr)
$$
and noticing that  $ F_{xx} =2 u u_x $,
$$
2iu P_{\neq c}Ê\Bigl( u\Hl\partial_xu\Bigr)-F_{xx}u= -4 u^2 P_{-}
u_{x}?2i u P_{c} (u\Hl u_x) \; .
$$
Moreover, following \cite{MR}, we will use the symmetry of the term $\partial_x^{-1} (u_x \Hl u_x)$. Indeed, it is easy to check that
 $ \partial_x^{-1} (u_x \Hl u_x)=-i \partial_x^{-1} (P_+ u_x)^2+i \partial_x^{-1} (P_- u_x)^2 $ and thus setting
  \begin{equation} \label{defB}
  B(u,v) =-i \partial_x^{-1} (P_+ u_x P_+ v_x)+i \partial_x^{-1} (P_- u_xP_- v_x) \; ,
  \end{equation}
  we infer that  $ \partial_x^{-1} (u_x \Hl u_x)=B(u,u) $.
 We thus finally get
\begin{align*}
I=P_+\Bigl[e^{-iF}&\Bigl(-4 u^2 P_{-} u_{x} -2i u  B(u,u)+2iu P_{c} (u\Hl u_x) -iu P_{c} \Bigl( (P_{\neq c} u^2)^2 \Bigr)  \Bigr]
\end{align*}
which  leads to
\begin{align}\label{eq:gauge1eq}
v_t-iv_{xx}=&P_+\Bigl[e^{-iF}\Bigl(-4 u^2 P_{-} u_{x} -2i P_{-} u_{xx}-2iu B(u,u)\nonumber\\
&-2iu P_{c} (u\Hl u_x) +iu P_{c} \Bigl( (P_{\neq c} u^2)^2 \Bigr)  \Bigr)\Bigr].
\end{align}
 Due to the projector $P_+,P_-$, we see formally that in the system
\eqref{eq:gauge}-\eqref{eq:gauge1eq} there is no high-low
interaction of the form
\[P_{low} u^2 \cdot \partial_xP_{high}u.\]
 Note that  $u\to \gau(u)$ can be "inverted" in Lebesgue space. This
is the strategy used in \cite{MRmBO} to prove well-posedness in
$H^1$. To go below to $H^{1/2}$, we intend to use $X^{s,b}$ spaces. But
$u\to \gau(u)$ can not be well "inverted" in Bourgain'spaces and  thus $ u $ will not have the same regularity as $ \gau(u)$ in these spaces.
To handle this former difficulty, we will insert the "inverse" into some of the
terms in \eqref{eq:gauge1eq}. We first observe that
\begin{align*}
-2iP_{+} \Bigl( e^{-iF}  P_{-} u_{xx}\Bigr) = & -2i \partial_x P_{+} ( e^{-iF} P_{-} u_x ) + 2 P_+ (e^{-iF} P_{\neq c} (u^2) P_{-} u_x) \\
= & -2 \partial_x P_+(\partial_x^{-1}P_{+}(e^{-iF} P_{\ne
c}(u^2))P_-u_{x})\\
&+2P_+(e^{-iF} u^2  P_-u_x)-2P_{c}(u^2) P_+(e^{-iF} P_-u_x)
\end{align*}
and thus the sum of the first two terms of the right-hand
side of \eqref{eq:gauge1eq}  can be rewritten as
\begin{align}
&-2P_+(e^{-iF} u^2  P_-u_x)-2 \partial_x P_+(\partial_x^{-1}P_{+}
(e^{-iF} u^2)P_-u_{x})\nonumber \\
& +2 P_{c}(u^2) \Bigl(  \partial_x P_+(\partial_x^{-1}P_{+} e^{-iF}
P_-u_{x}) - P_+(e^{-iF}   P_-u_x)\Bigr) \; . \label{M}
\end{align}
Now, let us denote
\begin{align}
R(u)& :=[P_+,e^{-iF}]u=P_+(e^{-iF}u)-e^{-iF}P_+u\nonumber \\
 & = P_+(e^{-iF} P_{-} u)+P_+(e^{-iF} P_{c} u) - P_{-} (e^{-iF} P_+ u)-P_{c} (e^{-iF} P_+ u)   \; . \label{eq:Ru}
\end{align}
Formally, $R(u)$ is a commutator, and has one
order higher regularity than $F_x=P_{\neq c} u^2 $ (see Lemma \ref{lemma2}).
Then we get
\begin{align}
v=e^{-iF}P_+u+R(u), \label{uu}
\end{align}
and thus
$ P_+ u = e^{iF} v -e^{iF}[P_+,e^{-iF}] u $. Since $ u$ is real-valued, this leads to
\begin{equation}
P_- u =  P_-(e^{-iF} \overline{v})-P_-(e^{-iF}\overline{R(u)}) \; .
\label{Pu}
\end{equation}
Substituting $ P_{-} u $ by the expression \eqref{Pu} in the two first terms of \eqref{M} we eventually get
the following equation satisfied by $ v $ :
\begin{align}
v_t-iv_{xx}=& 2 N^0(u,v)+2 N^1(u,v) -2i   P_+(e^{-iF}u B(u,u))+G(u)\label{eq:sys1}
\end{align}
where
$$
 N^\nu(u,v):= -\partial_x^{\nu}  P_+(\partial_x^{-\nu}P_{+} (e^{-iF}u^2) \partial_x P_-(e^{-iF}\overline{v})), \quad
 \nu=0,1.
 $$
and
\begin{align}
G(u) :=&P_+\Bigl( e^{-iF}\Bigl(-2i u P_{c} (u\Hl u_x)+u P_{c} ( (P_{\neq c} u^2)^2 ) +2 u^2 \partial_x P_{-} (e^{-iF} \overline{R(u)})\Bigr)\Bigr) \nonumber \\
& -2 \partial_x P_{+} \Bigl( \partial_x^{-1}P_{+}(e^{-iF}) P_-u_{x}\Bigr)
+2i  \partial_x P_{+}\Bigl( (e^{-iF}\partial_x P_{-} (e^{-iF} \overline{R(u)})\Bigr) \nonumber \\
 &+2P_{c}(u^2)  \Bigl(-P_+(e^{-iF} P_{-} u_{x} ) + \partial_x P_+ \Bigl( \partial_x^{-1} P_{+}
(e^{-iF}) P_-u_{x}) \Bigr) \Bigr) \; . \label{defG}
\end{align}
We will see that the worst terms of the right-hand side of
\eqref{eq:sys1} are the first two terms. Actually the use of Bourgain's spaces will be necessary to handle the first three terms of
\eqref{eq:sys1}. On the other hand, $ G(u) $
is a nice term that belongs to $ L^2_t
H^{1/2}_x $ as soon as $ u \in L^\infty_t H^{1/2} $.

\section{The main estimates and proof of Theorem \ref{thm}}

In this section, we present the main estimates. By combining all
these estimates, we finish the proof of Theorem \ref{thm}.

\subsection{Linear Estimates}
We list some linear estimates in this subsection. The first ones are
the standard estimates for the linear solution, see \cite{Bour} and  \cite{Gin}.
\begin{lemma}\label{lem:linear}
Let $s\in \R$. There exists $C>0$ such that for all $f\in
X^{s,-\half1+}$ and  all $u_0\in H^s$ we have
\begin{align}
\|W(t)u_0\|_{Y_T^s}\leq& C\|u_0\|_{H^s} \label{to1} \\
\normo{\int_0^tW(t-\tau)f(\tau)d\tau}_{Y_T^s}\leq&
C\|f\|_{X_T^{s,-\half1+}}\; . \label{to2}
\end{align}
\end{lemma}

Next, we need some embedding properties of the space $Y^s$. The
first one is the well-known estimate due to Bourgain \cite{Bour}
\begin{align}\label{str}
\|v\|_{L^4_{t,x}}\les\|v\|_{\widetilde{L^4_{t,x}}}\lesssim\|v\|_{X^{0,3/8}}
\end{align}
where the first inequality above follows from the Littlewood-Paley
square function theorem. Note that \eqref{to1} combined with
\eqref{str} ensures that for $ 0\le T \le 1 $,
\begin{equation}\label{to3}
\| W(t) u_0\|_{L^4_{tx}} \lesssim \| u_0\|_{L^2}\; .
\end{equation}

\subsection{Main Non-linear Estimates}

\begin{proposition}[Estimates of $u$]\label{lem:estu}
Let $ T\in ]0,1[ $, $s\in[\half1,1]$ and  $ (u_i,v_i) \in \Bigl( C^{0}_T H^{s}\cap \widetilde{L^{4}_TH_4^{s} }\Bigr) \times Y_T^s $, $ i=1,2 $,  satisfying
\eqref{eq:mBO2} and \eqref{eq:gauge}  on $ ]-T,T[ $ with initial data $ u_{i,0}$. Then for
$u=u_i$
\begin{align}\label{eq:lemestu2}
\|u\|_{\widetilde{L^{4}_TH^{s}_4}} &\lesssim
(1+\|u\|_{L^{\infty}_TH^{\frac{1}{2}}
}^4) \|v\|_{Y_T^{s}}
+T^\frac{1}{4}(1+\|u\|_{L^{\infty}_TH^{\frac{1}{2}} }^8) \|u\|_{L^\infty_T H^{\frac{1}{2}}},
\end{align}
and for large $N\in \N$, we have
\begin{align}
\|u\|_{L^{\infty}_TH^s}
\les&\|u_0\|_{H^s}+TN^2\|u\|_{L^\infty_TH^{1/2}}^3+\norm{v}_{Y^s_T}\nonumber\\
&+N^{-\frac{1}{4}}(\|u\|_{L^{\infty}_TH^{\frac{1}{2}}}+\norm{v}_{Y_T^s})(1+\|u\|_{L^{\infty}_TH^{\frac{1}{2}} }^8). \label{eq:lemestu3}
\end{align}
Moreover, we have
\begin{align}
\|u_1-u_2\|_{\widetilde{L^{4}_TH^{1/2}_4}}\les&(1+\|u\|_{L^{\infty}_TH^{1/2}
}^4) \|v_1-v_2\|_{Y_T^{\frac{1}{2}}}\nonumber\\
& +\norm{u_1-u_2}_{L_T^\infty
H^{1/2}}  \| v_1\|_{Y_T^{1\over 2}} \prod_{i=1}^2(1+\norm{u_i}_{L_T^\infty H^{\frac{1}{2}}})^3 \nonumber\\
&+T^{1/4}\norm{u_1-u_2}_{L_T^\infty H^{\frac{1}{2}}}\prod_{i=1}^2(1+\norm{u_i}_{L_T^\infty
H^{\frac{1}{2}}})^8 \label{dif1}
\end{align}
and
\begin{align}
\|u_1-u_2\|_{L^{\infty}_TH^{1/2}} \les &
\norm{u_{1,0}-u_{2,0}}_{H^{1/2}}+\prod_{i=1}^2(1+\norm{u_i}_{L_T^\infty
H^{\frac{1}{2}}}+\norm{v_i}_{Y^{1/2}_T})^8\nonumber\\
&\times\bigg(\norm{v_1-v_2}_{Y_T^{1/2}}+(T
N^2+N^{-1/4})\norm{u_1-u_2}_{L_T^\infty H^{\frac{1}{2}}}\bigg).
\label{dif2}
\end{align}
\end{proposition}

\begin{proposition}[Estimates of $v$]\label{lem:estv} Let  $ 0<T<1$, $ s\in[\frac{1}{2},1] $ and $(u_i,v_i)\in \Bigl( C^{0}_tH^{s}\cap \widetilde{L^{4}_TH_4^{s} }\Bigr) \times Y_T^s $ satisfying  \eqref{eq:mBO2}, \eqref{eq:gauge}  and \eqref{eq:sys1}  on $ ]-T,T[ $. Then for $(u,v)=(u_i,v_i) $ there exists $ \nu>0 $ and $ q\in \N^* $ such that
\begin{align}
\|v\|_{Y_T^{s}}\les &(1+\|u_0\|_{H^{\frac{1}{2}}}^4) \|u_0\|_{H^s}
+T^{\nu} \Bigl( (1+\|u\|_{L^\infty_T H^{1\over 2}\cap \widetilde{L^4 H^{1/ 2}_4}}^{q+1}) \| v\|_{X^{s,1/2}} \nonumber \\
& +   (1+\|u\|_{L^\infty_T H^{1\over 2}\cap \widetilde{L^4 H^{1/2}_4}}^{q})\| v\|_{X^{1/2,1/2}} \|u\|_{L^\infty_T H^{s}\cap \widetilde{L^4 H^{s}_4}}\Bigr)
.            \label{lem:estv1}
\end{align}
and
\begin{align}
\|v_1-v_2\|_{Y_T^{\frac{1}{2}}}\les &(1+\|u_0\|_{H^{\frac{1}{2}}}^4)
\|u_{1,0}-u_{2,0}\|_{H^s}\\
&+T^{\nu} \Bigl[ \Bigl(1+\sum_{i=1}^2 \|u_i\|_{L^\infty_T H^{1\over 2}\cap \widetilde{L^4 H^{1/ 2}_4}}^{q+1}\Bigr) \| v_1-v_2\|_{X^{s,1/2}} \nonumber \\
& +   \Bigl(1+\sum_{i=1}^2\|u_i\|_{L^\infty_T H^{1\over 2}\cap
\widetilde{L^4 H^{1/2}_4}}^{q}\Bigr)(\sum_{i=1}^2\|
v_i\|_{X^{1/2,1/2}} )\|u_1-u_2\|_{L^\infty_T H^{s}\cap
\widetilde{L^4 H^{s}_4}}\Bigr] .            \label{lem:estv11}
\end{align}
\end{proposition}

The rest of this subsection is devoted to proving Proposition
\ref{lem:estu}, while the proof of Proposition \ref{lem:estv} will
be given in the next section.
\subsection{Proof of  Proposition
\ref{lem:estu}}
We start with recalling some  technical
lemmas that  will be needed hereafter.
We first recall the Sobolev multiplication
laws.
\begin{lemma}\label{lem:mul}
(a) Assume one of the following condition
\[s_1+s_2\geq 0, s\leq s_1,s_2,s<s_1+s_2-\frac{1}{2},\]
\[ \mbox{ or }\
s_1+s_2>0,s<s_1,s_2, s\leq s_1+s_2-\frac{1}{2}.\] Then
\[\norm{fg}_{H^s}\les \norm{f}_{H^{s_1}}\norm{g}_{H^{s_2}}.\]

(b) For any $s\geq 0$, we have
\[\norm{fg}_{H^s}\les \norm{f}_{H^s}\norm{g}_{L^\infty}+\norm{g}_{H^s}\norm{f}_{L^\infty}.\]
\end{lemma}
Second,  we state the
classical fractional Leibniz rule estimate derived by Kenig, Ponce
and Vega (See Theorems A.8 and A.12 in \cite{KPV2}).
\begin{lemma} \label{leibrule}
Let $0<\alpha<1$, $p, \ p_1, \ p_2 \in (1,+\infty)$ with $\frac1{p_1}+\frac1{p_2}=\frac1p$
and $\alpha_1, \ \alpha_2 \in [0,\alpha]$ with $\alpha=\alpha_1+\alpha_2$. Then,
\begin{equation} \label{leibrule1}
\big\|D^{\alpha}_x(fg)-fD^{\alpha}_xg-gD^{\alpha}_xf \big\|_{L^p}
\lesssim \|D_x^{\alpha_1}g\|_{L^{p_1}}\|D^{\alpha_2}_xf\|_{L^{p_2}}.
\end{equation}
Moreover, for $\alpha_1=0$, the value $p_1=+\infty$ is allowed.
\end{lemma}
The next estimate is a frequency localized version of estimate \eqref{leibrule1}
in the same spirit as Lemma 3.2 in \cite{Tao}. It allows to share most of the fractional
derivative in the first term on the right-hand side of \eqref{lemma2.1}.
\begin{lemma} \label{lemma2}
Let $ \alpha, \beta  \ge 0$ and $1<q<\infty$. Then,
\begin{equation} \label{lemma2.1}
\big\|D_x^{\alpha}P_\mp\big(fP_\pm D_x^{\beta}g\big) \big\|_{L^q} \lesssim
\|D_x^{{\alpha}_1}f\|_{L^{q_1}}\|D_x^{{\alpha}_2}g\|_{L^{q_2}},
\end{equation}
with $1<q_i<\infty$, $\frac1{q_1}+\frac1{q_2}=\frac1q$ and $\alpha_1
\ge \alpha$, $\alpha_2 \ge 0$ and $\alpha_1+\alpha_2=\alpha+\beta $.
\end{lemma}
\begin{proof}
See Lemma 3.2 in \cite{Mol1}.
\end{proof}
Finally we state the two  following lemmas. The first one is a direct consequence of the continuous embeddings
 $H^{s+1/4} \hookrightarrow H^{1/2}_4 \hookrightarrow L^\infty $ whereas the proof of the second one (in the real line case) can be found in [\cite{MR}, Lemma 6.1].
\begin{lemma} \label{lemma4}
Let $ s\in [1/2,1]$, $ z\in  L^\infty_T H^{s+\frac{1}{4}} $ and $ v\in \widetilde{L^4_T H^s_4} $ then
\begin{equation}
\| z v \|_{ \widetilde{L^4_T H^s_4}} \lesssim \| z \|_{  L^\infty_T H^{s+\frac{1}{4}+}} \|v\|_{ \widetilde{L^4_T H^s_4}}
\end{equation}
\end{lemma}
\begin{lemma} \label{lemma5}
Let $ v_1, v_2 \in \widetilde{L^4 H^{1/2}_4} $ then
\begin{equation}
\|B(v_1,v_2) \|_{L^2} \lesssim \| D_x^{1/2} v_1 \|_{L^4}  \| D_x^{1/2} v_2 \|_{L^4}
\end{equation}
\end{lemma}
Let $ k\in \Z^* $ with $ |k|\le 10 $. A direct computation gives
\begin{align}
\partial_x
(e^{ikF})&=kie^{ikF}(u^2-P_c(u^2)), \label{de1}
\end{align}
Next by gathering  the obvious estimates  $\|e^{ikF}\|_{L_T^\infty L_x^2}\les 1$ and  $\|\partial_x
(e^{ikF})\|_{L^{\infty}_TL^2_x}\les P_c(u^2)+\|u\|^2_{L^{\infty}_TL^4_x}$,  we get
\begin{align}
\|e^{ikF}\|_{L^{\infty}_TH^1}&\les 1+\|u\|^2_{L^{\infty}_TH^\half1}.
\label{24}
\end{align}
On the other hand, by Lemma \ref{lem:mul}, we have for any $ s\in [1/2,1] $,
\begin{align*}
\|\partial_x (e^{ikF})\|_{L^{\infty}_TH^{s-}} &\les
\norm{e^{ikF}}_{L_T^\infty
H^1}\norm{u^2-P_c(u^2)}_{L^{\infty}_TH^{s-}}\les
(1+\|u\|_{L^{\infty}_TH^{\frac{1}{2}}}^2) \| u\|_{L^{\infty}_T
H^{\frac{1}{2}}} \|u\|_{L^{\infty}_T H^s}\; .
\end{align*}
Gathering the above estimates leads for any $ s\in [1/2,1]$  to
\begin{align}
\|e^{ikF}\|_{L^{\infty}_TH^{s+1-}} & \les
(1+\|u\|_{L^{\infty}_TH^{\frac{1}{2}}}^2) \Bigl(1+\| u\|_{L^{\infty}_T H^{\frac{1}{2}}} \|u\|_{L^{\infty}_T H^s}\Bigr) \label{24b}
\end{align}
and, in view of \eqref{eq:Ru} and  Lemma \ref{lemma2},  it holds
\begin{align}
\| R(u)\|_{L^\infty_T H^{\frac{5}{4}+}} & \lesssim \|e^{iF} \|_{L^\infty_T H^{{5\over 4}+}_{2+}} \|u\|_{L^\infty_T L^{\infty -}_x} \\
& \lesssim \|e^{-iF}\|_{L^{\infty}_TH^{\frac{3}{2}-}} \| u\|_{L^\infty_T H^{\frac{1}{2}}}\nonumber \\
 &\lesssim  (1+\| u\|_{L^\infty_T H^{\frac{1}{2}}}^4) \| u\|_{L^\infty_T H^{\frac{1}{2}}} \; .\label{Ru}
\end{align}
Now, for
$s\in [1/2,1]$,  according to \eqref{uu}, \eqref{24b}-\eqref{Ru} and Lemma \ref{lemma4}?we easily  get
\begin{align*}
\|P_+u\|_{\widetilde{L^{4}_TH^{s}_4}} &\lesssim
\|e^{iF}v\|_{\widetilde{L^{4}_TH^{s}_4}}
+\|e^{iF}R(u)\|_{\widetilde{L^{4}_TH^{s}_4}}
\\
&\lesssim
\|e^{iF}\|_{L^{\infty}_tH^{3/2-}}\|v\|_{\widetilde{L^{4}_TH^{s}_4}}
+T^\frac{1}{4}\|e^{iF}\|_{L^{\infty}_tH^{3/2-}}\|R(u)\|_{L^\infty_tH^{{5\over 4}+}} \\
& \lesssim    (1+\| u\|_{L^\infty_T H^{\frac{1}{2}}}^4)\Bigl( \|v\|_{Y^s_T}+T^\frac{1}{4} (1+\| u\|_{L^\infty_T H^{\frac{1}{2}}}^4)\| u\|_{L^\infty_T H^{\frac{1}{2}}} \Bigr)
\end{align*}
Estimate \eqref{eq:lemestu2} follows by using that $ u $ is real valued and  the conservation  of the mean-value by  \eqref{eq:mBO2}.\\
Next, in order to get a better estimate of
$\|u\|_{L^\infty_TH^s},\ s\in [\half1,1]$, we split $u$ into
a  low frequency and a high frequency part. For low frequency,
we use the equation for $u$, while for high frequency, we use
$P_+u=e^{iF}v-e^{iF}R(u)$. For any $N=2^k\in \N$, and
$s\in[\half1,1]$, we have
\begin{align*}
\|u\|_{L^\infty_TH^s}\lesssim\|P_{\leq k}u\|_{L^\infty_TH^s}
+2 \|P_{\geq k}^+u\|_{L^\infty_TH^s}.
\end{align*}
By the equation of $u$, we have
\begin{align*}
P_{\leq k}u=W(t)P_{\leq k}u_0 +\frac{1}{3}\int_0^t
W(t-\tau)P_{\leq k}\partial_x(u^3)(\tau)d\tau,
\end{align*}
that leads to
\begin{align*}
\|P_{\leq k} u\|_{L^\infty_TH^s}\les \|u_0\|_{H^s}
+T2^{2k}\|u\|_{L^\infty_TH^{1/2}}^3.
\end{align*}

To estimate the term $\|P_{\geq k}^+u\|_{L^\infty_TH^s}$, we use
$P_+u=e^{iF}v-e^{iF}R(u)$. By  \eqref{24} -\eqref{Ru} we have
\begin{eqnarray*}
\|P_{\geq k}^+[e^{iF}R(u)]\|_{L^\infty_TH^s} & \les& N^{-1/4}Ê\| e^{-iF} R(u) \|_{L^\infty_T H^{5/4}} \\
& \les&N^{-\frac{1}{4}}(1+\|u\|_{L^{\infty}_TH^{\frac{1}{2}}}^8) \|u\|_{L^\infty_T H^{1/2}} \; .
\end{eqnarray*}
It remains to estimate $\|P_{\geq k}^+[e^{iF}v]\|_{H^{s}}$. By Lemma \ref{lem:mul} we have
\begin{align*}
\|P_{\geq k}^+[e^{iF}v]\|_{L^\infty_tH^s}\les&\|P_{\geq k}^+[P_{\leq
k-5}(e^{iF})v]\|_{L^\infty_tH^s}+\|P_{\geq k}^+[P_{>
k-5}(e^{iF})v]\|_{L^\infty_tH^s}\\
\les& \norm{e^{iF}}_{L^\infty_T L^\infty_x} \norm{v}_{L^\infty_tH^s}+\norm{v}_{L^\infty_tH^s}\norm{P_{\geq
k-5}(e^{iF})}_{L^\infty_tH^{1}}\\
\les&\norm{v}_{L^\infty_tH^s} \Bigl( 1+N^{-1/4}(1+\|u\|_{L^{\infty}_TH^{\frac{1}{2}}}^4)\Bigr) .
\end{align*}
Then $(\ref{eq:lemestu3})$ holds.
For the difference estimates \eqref{dif1}-\eqref{dif2}, the proofs are  similar.
We only need to observe that by the mean-value theorem,
$ |e^{ikF(u_1)}-e^{ikF(u_2)}|\le | k (P_{\neq c} (u_1^2-u_2^2) |$ and thus
\begin{equation} \label{za1}
\|e^{ikF(u_1)}-e^{ikF(u_2)}\|_{L_T^\infty L_x^2}\les
\norm{u_1-u_2}_{L_T^\infty H^{1/2}}(\norm{u_1}_{L_T^\infty
H^{1/2}}+\norm{u_2}_{L_T^\infty H^{1/2}}) \;
\end{equation}
and
\begin{align}
\|\partial_x(e^{ikF(u_1)}-e^{ikF(u_2)})\|_{L_T^\infty L_x^2} & \les
\norm{u_1-u_2}_{L_T^\infty H^{1/2}}(\norm{u_1}_{L_T^\infty
H^{1/2}}+\norm{u_2}_{L_T^\infty H^{1/2}}) \nonumber \\
&  + \| P_{\neq c} (u_1^2) (e^{ikF(u_1)}-e^{ikF(u_2)})\|_{L_T^\infty L_x^2}\nonumber \\
 \les &
\norm{u_1-u_2}_{L_T^\infty H^{1/2}}(\norm{u_1}_{L_T^\infty
H^{1/2}}+\norm{u_2}_{L_T^\infty H^{1/2}})^3 \; .\label{za2}
\end{align}
\subsection{Proof of Theorem \ref{thm}}

In this subsection, we prove Theorem \ref{thm}. We will rely on
the results obtained in \cite{MRmBO}:
\begin{lemma}[\cite{MRmBO}]\label{lem:energyres}
The mBO equation \eqref{eq:mBO} is locally well-posed in $H^s$ for
$s\geq 1$. Moreover, the minimal  length of the interval of existence  is
determined by $\norm{u_0}_{H^1}$.
\end{lemma}
Now, fixing any $u_0\in H^{1/2}(\T)$, we choose $\{u_{0,n}\}\subset
C^{\infty}(\T)$, real-valued, such that $u_{0,n}\to u_0$ in
$H^{1/2}$. We denote by $u_n$ the solution of mBO emanating from
$u_{0,n}$ given by Lemma \ref{lem:energyres} and
$v_n=P_+(e^{-iF(u_n)}u_n)$.

{\bf Step 1.} A priori estimate: we show that there exists
$T=T(\norm{u_0}_{H^{1/2}})>0$ such that $u_n$ exists on $(-T,T)$.

It suffices to show that there exists a
$T=T(\norm{u_0}_{H^{1/2}})>0$, such that for any $n\in \N$, if
$|t|\leq T$ and $u_n(t)$ exists, then
\begin{align}\label{eq:apriori}
\norm{u_n(t)}_{H^s}\leq C(\norm{u_{0,n}}_{H^s}),\quad 1/2\leq s\leq
1.
\end{align}
First we show \eqref{eq:apriori} for $s=1/2$. We may assume
$\norm{u_{0,n}}_{H^{1/2}}\leq 2\norm{u_{0}}_{H^{1/2}}$, $\forall
n\in \N$. Define the quantity $\norm{(u,v)}_{F^s_T}$ by
\[\norm{(u,v)}_{F^s_T}:=\norm{u}_{L_T^\infty H^{s}}+\norm{v}_{X^{s,\frac{1}{2}}_T}.\]
Applying Proposition \ref{lem:estu}-\ref{lem:estv} to $u_n,v_n$
(taking $s=1/2$), we get
\begin{align*}
\norm{(u_n,v_n)}_{F^{1/2}_T}\les&
(1+\|u_{0}\|_{H^{1/2}}^8) \|u_{0}\|_{H^{1/2}} +(T^{1/4}N^2+N^{-1/4})\norm{(u_n,v_n)}_{F^{1/2}_T}^9\\
& + T^\nu \Bigl(1+ \norm{(u_n,v_n)}_{F^{1/2}_T}^{k} \Bigr) \norm{(u_n,v_n)}_{F^{1/2}_T}\; ,
\end{align*}
for some $ \nu>0 $ and $ k\in \N^* $ and for any $ N\ge 1$ and $ 0<T<1$.
Therefore taking $ N $ large enough, we infer that there exits $
T=T(\norm{u_0}_{H^{1/2}})>0$ such that \eqref{eq:apriori} holds for $ s=1/2$.
Now, for  $1/2<s\leq 1$, we have
\begin{align*}
\norm{(u_n,v_n)}_{F^{s}_T}\les&
(1+\|u_{0}\|_{H^{1/2}}^8) \|u_{0,n}\|_{H^{s}} +(T^{1/4}N^2+N^{-1/4})\norm{(u_n,v_n)}_{F^{1/2}_T}^8\norm{(u_n,v_n)}_{F^{s}_T} \\
& + T^\nu \Bigl(1+ \norm{(u_n,v_n)}_{F^{1/2}_T}^{k} \Bigr) \norm{(u_n,v_n)}_{F^{s}_T}\; ,
\end{align*}
which yields   \eqref{eq:apriori} for some $T=T(\norm{u_0}_{H^{1/2}})>0$ smaller if necessarily. This completes the Step 1.

{\bf Step 2.} Next, we will show that $u_n$ is a Cauchy sequence in
$C([-T,T];H^{1/2})$.

Applying the difference estimates in Proposition
\ref{lem:estu}-\ref{lem:estv} to $(u_n,v_n)$, arguing as in Step 1,
we get
\begin{align}
\|(u_n-u_m,v_n-v_m)\|_{F_T^{1/2}}\les\|u_{0,n}-u_{0,m}\|_{H^{\frac{1}{2}}}. \label{qq}
\end{align}
Thus, ${(u_n,v_n)}$ is a Cauchy sequence, and there exists $u\in
C([-T,T];H^{1/2})$ such that
$\|u_n-u\|_{L^{\infty}_TH^{\frac{1}{2}}}\rightarrow 0,\ n\rightarrow
\infty$. By classical compactness  arguments, it is easy to check that $ u $  solves the mBO equation. Moreover, in view of \eqref{qq}  it is the only solution in the class $ u\in L^\infty_T H^{1/2} $ with $ P_+(e^{iF(u)}) \in X^{\frac{1}{2},\frac{1}{2}}_T $ and the solution-map $ u_0 \mapsto u $ is continuous from $ H^{\frac{1}{2}}(\T) $ into $C([-T,T];H^{1/2})$.
 At last, in the defocusing case using the conservation
of $H^{\frac{1}{2}}$ norm of $u$, we get that  $u $ is global in time.

\section{Proof of the estimates on $v$}
In this section, we prove Proposition \ref{lem:estv}. We will work
on the equation \eqref{eq:sys1}. By Lemma
\ref{lem:linear} and the trivial embedding $ L^2_T H^s \hookrightarrow X_T^{s,-\half 1+}$, we infer that
\begin{align}
\norm{v}_{Y_T^s} \les& \norm{v(0)}_{H^s}+T^\nu \Bigl(
\norm{G(u)}_{L_T^2H^s}+\norm{N^0(u,v)}_{X_T^{s,-\half 1+}}
+\norm{N^1(u,v)}_{X_T^{s,-\half 1+}}\nonumber\\
&+\norm{P_+[-2ie^{-iF}uB(u,u)]}_{X_T^{s,-\half 1+}}\Bigr)
\label{est}
\end{align}
 for some  $ \nu>0 $.
Then to prove Proposition \ref{lem:estv}, we will estimate the terms of the
right-hand side one by one.
\subsection{Estimate on $ G(u)$}
\begin{lemma}
Let $ 1/2\le s\le 1 $, $0<T\leq 1$ and  $u_i\in C([-T,T]:H^{s})\cap \widetilde{L_T^4H_4^{s}}$, $ i=1,2$,
be two solutions to  \eqref{eq:mBO2} with initial data $ u_{i,0}$. Then for $u=u_i$ we have
\begin{align*}
\norm{\gau(u_{i,0})}_{H^s}\les& (1+\norm{u_{i,0}}_{H^{1/2}}^4)\norm{u_{i,0}}_{H^s}\\
\norm{G(u_i)}_{L_T^2H^s}\les& (1+\norm{u_i}_{L_T^\infty
H^{1/2}\cap L_T^4H_4^{1/2}}^{12}) \norm{u_i}_{L_T^\infty H^{s}\cap
L_T^4H_4^{s}}\; .
\end{align*}
Moreover, it holds
\begin{align*}
\norm{\gau(v_{1,0})-\gau(v_{2,0})}_{H^{1/2}}\les& \norm{u_{1,0}-u_{2,0}}_{H^{1/2}}\prod_{i=1}^2(1+\norm{u_{i,0}}_{H^{1/2}}^4)\\
\norm{G(u_1)-G(u_2)}_{L_T^2H^{1/2}}\les&
\norm{u_1-u_2}_{L_T^\infty H^{1/2}\cap
L_T^4H_4^{1/2}}\prod_{i=1}^2(1+\norm{u_i}_{L_T^\infty H^{1/2}\cap
L_T^4H_4^{1/2}}^{12}).
\end{align*}
where the gauge transformation $ \gau $ and  the function $  G $ are  defined respectively in \eqref{eq:gauge} and \eqref{defG}.
\end{lemma}
\begin{proof}
The estimates on $v_{i,0}=\gau(u_{i,0})$ and its difference are similar to the
estimates of $u$ in the proof of Proposition \ref{lem:estu}. The estimates on $ G$ follow from
 the definition \eqref{defG} of $G(u)$,  Lemma \ref{lem:mul} and Lemma \ref{lemma2}. For instance we have
 \begin{align*}
\|e^{-iF} P_{+}(u^2 \partial_x P_{-} (e^{-iF} \overline{R(u))}
\|_{H^s} \lesssim &   \|e^{-iF} \|_{H^{\frac{3}{2}-}} \|
u^2\|_{H^s_4}
 \|\partial_x P_{-} (e^{-iF} \overline{R(u)} )\|_{H^{0+}_4} \\
\lesssim &   \|e^{-iF} \|_{H^{\frac{3}{2}-}}^2  \|u\|_{H^s_4} \|u\|_{H^{1\over 2}_4} \|R(u)\|_{H^{{5\over 4}+}} \\
\lesssim &  \|u\|_{H^s_4} \|u\|_{H^{1\over 2}_4}\|u\|_{H^{1\over 2}}
(1+\|u\|_{H^{1/2}}^4)^3
\end{align*}
and
$$
 \Bigr\|\partial_x P_+ \Bigl(\partial_x^{-1} P_+ (e^{-iF} ) P_{-} u_x\Bigr) \Bigl\|_{H^s}  \lesssim
  \|e^{-iF}Ê\|_{H^1_{4}} \|u\|_{H^s_4}Ê\lesssim  (1+\|u\|_{H^{1/2}}^4)  \|u\|_{H^s_4}
 $$
\end{proof}
\subsection{Estimates on  suitable  extensions of $ u $ and  $ e^{-iF(u)}$ .}
Before proving the main multilinear estimates, we need to prove estimates on suitable extensions of $ u $ and  $ e^{-iF(u)}$.
\begin{lemma}\label{lemma42}
Let $1/2\le  s\le 1$, $0<T\leq 1$ and  $u_1,u_2\in C([-T,T]:H^{s})\cap \widetilde{L_T^4H_4^{s}}$
be two solutions to  \eqref{eq:mBO2}. Then for $i=1,2$
\begin{align}
\|u_i\|_{(X^{s-1,1} \cap L^\infty_t H^s\cap \widetilde{L^4_t H^s_4})_T}&\le
(1+\|u_i\|_{L^{\infty}_TH^{\frac{1}{2}}_x}^2) \|u_i\|_{L^\infty_T H^s \cap \widetilde{L^{4}_TH^{s}_4}}. \label{tt}
\end{align}
Moreover, we have
\begin{align}
\|u_1-u_2\|_{(X^{-\frac{1}{2},1}\cap L^\infty_t H^\frac{1}{2}\cap \widetilde{L^4_t H^\frac{1}{2}_4})_T }&\les\|u_1-u_2\|_{L^{\infty}_TH^{1/2}\cap
\widetilde{L_T^4H_4^{1/2}}}\prod_{i=1}^2(1+\|u_i\|_{L^{\infty}_TH^{1/2}\cap
L_T^4H_4^{1/2}}^2).
\end{align}
\end{lemma}
\begin{proof}
We consider $w(t)=W(-t)u(t)$ on the time interval $[-T,T]$ and
extend $w$ on $(-2,2)$ by setting $\partial_t w=0$ on
$[-2,-2]\setminus[-T,T]$. Then it is clear that for any $\theta\in \R $,
\[\norm{\partial_t w}_{L^2((-2,2):H^{\theta})}=\norm{\partial_t w}_{L^2_T
H^{\theta}},\ \norm{w}_{L^2((-2,2):H^{\theta})}\les\norm{w}_{L^\infty_T
H^{\theta}}\] Now we define $\tilde u(t)=\eta(t)W(t)w(t)$.  $\tilde
u$ is clearly an extension of $u$ outside $(-T,T)$ and  it suffices to prove \eqref{tt} with the $ X^{s-1,1}$, $ L^\infty_t H^s$ and $  L^4_t H^s_4$-norms of $ \tilde{u} $ in the left-hand side. First, using that  $\partial_t w=2 W(-t)(P_{\neq c} (u^2) u_x)$, we get
\begin{align*}
\norm{\tilde u}_{X^{s-1,1}} \lesssim & \norm{w}_{L^2((-2,2):H^{s-1})}+\norm{\partial_t w}_{L^2((-2,2):H^{s-1})} \\
\lesssim &  \norm{ u}_{L^2((-2,2):H^{s-1})}+ \| D_x^{s} (u^3)\|_{L^2_{Tx}} + \|u\|_{L^\infty_T L^2_x}^2 \|D_x^s u \|_{L^2_{Tx}}\\
\lesssim &  \norm{u}_{L^2((-2,2):H^{s-1})}+\|D_x^{s} u
\|_{L^4_{Tx}\cap L^\infty_T L^2_x}
\|u\|_{L^{\infty}_TH^{\frac{1}{2}}_x}^2
\end{align*}
where in the last step we used  Lemma \ref{leibrule} together with
$ L^{\infty}_t H^{1/2}_x \hookrightarrow L^8_{tx} $. Second,
     $$
 \norm{\tilde{u}}_{L^\infty_t H^s}  \lesssim  \| \eta(t) W(t) w(t) \|_{L^\infty_t H^s}? \lesssim \|w\|_{L^\infty_T H^{s}} \lesssim \|u\|_{L^\infty_T H^{s}}\; .
$$
Third, we notice that
$$
\norm{\tilde{u}}_{\widetilde{L^4_t H^s_4}} \lesssim \norm{u}_{\widetilde{L^4(]-T,T[ ;H^s_4)}} +\norm{W(t)w(t) }_{\widetilde{L^4(]-2,2[/]-T,T[;H^s_4)}}
$$
with $ w(t)=w(T) $ for all $ t\in ]T,2[ $ and $ w(t) =W(-T) $ for all $t\in ]-2,-T[ $. Therefore, in view of \eqref{to3},
$$
\norm{W(t)w(t) }_{\widetilde{L^4(]T,2[;H^s_4)}} =\norm{W(t)w(T) }_{\widetilde{L^4(]T,2[;H^s_4)}}\lesssim \|w(T)\|_{H^s} =\|u(T)\|_{H^s} \lesssim
 \|u\|_{L^\infty_T H^s} \; .
$$
This completes the proof of \eqref{tt}.
Finally the estimates for the difference is similar and thus will be omitted.
\end{proof}

Next, we prove the properties of the factor $e^{ikF}$.

\begin{lemma}\label{lemma43}
Let $1/2\le  s\le 1$, $0<T\leq 1$ and  $u_1,u_2\in C([-T,T]:H^{s})\cap \widetilde{L_T^4H_4^{s}}$
be two solutions to  \eqref{eq:mBO2}. Then for $i=1,2$\begin{align}\label{eq:lgf1}
\|e^{-iF(u_i)}\|_{(L^{\infty}_t H^{s+1-}\cap
X^{-\frac{1}{2}-,1})_T} \les 1+  ( 1
+\|u_i\|^6_{L^{\infty}_TH^{\frac{1}{2}}\cap L^4_T H^{1\over 2}_4}) \|u_i\|_{L^\infty_T H^s\cap L^4_T H^s_4} \; .
\end{align}
Moreover,
\begin{align}
\|e^{-iF(u_1)}-e^{-iF(u_2)} & \|_{(L^{\infty}_t H^{\half3-}\cap
L^{4}_t H^{\frac{3}{2}}\cap X^{-\frac{1}{2}-,1})_T} \nonumber\\
\les& \norm{u_1-u_2}_{L^{\infty}_TH^{\frac{1}{2}}\cap
L^{4}_TH^{\frac{1}{2}}_{4}}\prod_{i=1}^2(1
+\|u\|^6_{L^{\infty}_TH^{\frac{1}{2}}\cap L^{4}_TH_4^{\frac{1}{2}}}) \; . \label{dif}
\end{align}
\end{lemma}

\begin{proof}
  We  set  $ z(t)=W(-t) e^{-iF} $ on $]-T,T[ $ and than extend $ z $ on $]-2,2[ $ by setting
 $ \partial_t z =0  $  on
$[-2,-2]\setminus[-T,T]$. Then $\tilde  w =\eta(t) W(t) z(t) $
 is an extension of $e^{-iF}$ outside $(-T,T)$. As in the previous lemma, for any $ \theta\in \R $, it holds
 $$\|\tilde  w\|_{L^\infty_t H^\theta} \lesssim \| e^{-iF} \|_{L^\infty_T H^\theta}?$$
 which together with \eqref{24}-\eqref{24b} gives the estimate for the first term on the left-hand side
of \eqref{eq:lgf1}. Moreover,
$$
\| \tilde{w} \|_{X^{-\frac{1}{2}-,1 }} \lesssim \norm{e^{-iF}}_{L^2_TH^{-\frac{1}{2}-}} +
\norm{(\partial_t+\Hl \partial_{x}^2)e^{-iF}}_{L^2_TH^{-\frac{1}{2}-}}
$$.

with
\begin{align*}
(\partial_t+\Hl \partial_{x}^2) e^{-i F} & = -ie^{-iF}F_t -i \Hl  \Bigl( e^{-iF} \Bigl(2u u_x -i (P_{\neq c} (u^2) )^2 \Bigr)\Bigr) \\
\end{align*}
 According to the expression \eqref{F_t} of $ F_t $,  Lemma \ref{lem:mul} and Lemma \ref{lemma5} , it holds
 $$
 \| F_t\|_{L^2_TH^{-\frac{1}{2}-}}+\Bigl\| 2u u_x +ik (P_{\neq c} (u^2) )^2\Bigr\|_{L^2_TH^{-\frac{1}{2}-}}\lesssim \|u\|_{L^\infty_T H^{1/2}}^4 + \| u\|_{L^4_T H^{1/2}_4}^2
 $$
which yields the desired result by using \eqref{24} and again   Lemma \ref{lem:mul}. \\
For the difference estimate \eqref{dif}, the proof is similar by using \eqref{za1}-\eqref{za2}.
 The details are omitted.
\end{proof}
\subsection{Multilinear estimates}
With Lemmas \ref{lemma42}-\ref{lemma43} in hand,
the following  proposition enables us to treat the worst term of
\eqref{est}, that is $ N^{\nu}(u,v)$ with $\nu\in \{0,1\} $.
\begin{proposition} Let $ 1/2\le s \le 1 $, $ w_1,w_4 \in X^{-1/2-,1}\cap L^\infty_t H^{s+1-}_x  $ ,
$u_2,u_3\in X^{s-1,1}\cap L^\infty_t H^{s}_x \cap \widetilde{L^4_t
H^{s}_4} $ and $ v_5\in X^{1/2,1/2} $  with compact support in time.
Then it holds
\begin{align}
\Bigl\|&  \partial_x^{\nu} P_{+} \Bigl( \partial_x^{-\nu}P_{+}
\Bigl(w_1 u_2 u_3 \Bigr)  \partial_xP_{-}( w_4 v_5) \Bigr)
\Bigr\|_{X^{s,-{1\over 2}+}}
\nonumber \\
 \lesssim &\,   \| w_1 \|_{L^\infty_t H^{s+1-}} \|w_4\|_{L^\infty_T H^{1/2}_x} \|v_5\|_{X^{1/2,1/2}} \prod_{i=2}^3  \|u_i\|_{L^\infty_t H^{1/2}_x}\nonumber \\
&  +\|v_5 \|_{X^{1/2,1/2}}    \prod_{i=1,4}  \Bigl( \|w_i\|_{X^{-1/2-,1}\cap L^\infty_t H^{3/2-}_x} \Bigr)\times\nonumber\\
& \sum_{ 2\le i\neq j \le 3}   \Bigl( \|u_i\|_{X^{-1/2,1}\cap
L^\infty_t H^{1/2}_x\cap \widetilde{L^4_t H^{1/2}_x}}\Bigr)
 \Bigl( \|u_j\|_{X^{s-1,1}\cap L^\infty_t H^{s}_x\cap \widetilde{L^4_t
 H^{s}_4}}\Bigr).
\end{align}
\end{proposition}
\begin{proof}
 We want  to prove that
\begin{align*}
I:=&\Bigl\|  \partial_x^{\nu} P_{+}  \Bigl( \partial_x^{-\nu} P_{+} \Bigl(w_1 u_2 u_3\Bigr) \partial_xP_{-} (w_4 v_5) \Bigr)  \Bigr\|_{X^{s,-{1\over 2}+}}\\
= & \Bigl\| \sum_{N\ge 2, N_{123}\ge N , N_{45}\le N_{123}} \sum_{N_i, \; 1\le i\le 5} \\
& \quad \partial_x^{\nu} P_N P_{+}  \Bigl(
\partial_x^{-\nu}P_{N_{123}} \Bigl(P_{N_1}w_1 P_{N_2} u_2 P_{N_3}
u_3
   \Bigr) \partial_x P_{N_{45}}P_{-} (P_{N_4} w_4 P_{N_5} v_5) \Bigr)  \Bigr\|_{X^{s,-1/2+}} \; .
\end{align*}
     By the triangle inequality we can separate this sum in different sums on disjoint subset of $ (2^{\N})^8 $.
      By symmetry we can assume that $ N_2\le N_3 $.  \hspace*{2mm} \\
         {\bf 1.}    $N_4 \ge 2^{-8} N_5 $. Then $N_{45}\lesssim N_4 $ and we can write by almost orthogonality
         \arraycolsep1pt
         \begin{eqnarray*}
               I  & \lesssim  &\Bigl[ \sum_{N_{123}} \Bigl( \sum_{N_4}  \sum_{N_5\lesssim N_4}  \sum_{2\le N\le N_{123}} \sum_{ N_{45}\lesssim N_{4}}  \\
                & &\quad \quad  \Bigl\|\partial_x^{\nu}  P_N P_{+}
         \Bigl( P_{N_{123}}  \partial_x^{-\nu}\Bigl( w_1 u_2 u_3 \Bigr) \partial_x P_{N_{45}}P_{-}(P_{N_4} w_4 P_{N_5}v_5)\Bigr)
     \Bigr\|_{X^{s,-1/2+}}\Bigr)^2 \Bigr]^{1/2}  \\
  & \lesssim  &\Bigl[ \sum_{N_{123}} \Bigl( \sum_{N_4,N_5}   \sum_{2\le N\le N_{123}} \sum_{ N_{45}\lesssim N_{4}} N^s   \Bigl\|\partial_x^{\nu}  P_N P_{+}
         \Bigl( P_{N_{123}}  \partial_x^{-\nu}\Bigl( w_1 u_2 u_3 \Bigr) \\
         &&\quad\cdot\partial_x P_{N_{45}}P_{-}(P_{N_4} w_4 P_{N_5}v_5)\Bigr)
     \Bigr\|_{L^{4/3}_{tx}}\Bigr)^2 \Bigr]^{1/2}  \\
         & \lesssim  &\Bigl[ \sum_{N_{123}}\Bigl(N_{123}^{s}\|P_{N_{123}}(w_1 u_2 u_3) \|_{L^2_{tx}}\sum_{N\le N_{123}}(\frac{N}{N_{123}})^{s}
         \Bigr)^2 \Bigr]^{1/2} \\
         & & \sum_{N_4,N_5,N_{45}}N_{45}^{0-}N_4^{1+}N_5^{0-}\|P_{N_4}w_4 \|_{L^8_{tx}}\| P_{N_5}v_5\|_{L^8_{tx}} \\
         & \lesssim &  \|J^{s}_x(w_1 u_2 u_3) \|_{L^2_{tx}} \|w_4\|_{L^\infty_t H^{{3\over 2}-}_x}  \|v_5 \|_{L^8_{tx}}  \\
            & \lesssim &  (\|J_x^{s} w_1\|_{L^4_{tx}}+\|J_x^{s} u_2\|_{L^4_{tx}}+\|J_x^{s} u_3\|_{L^4_{tx}})
              (\|w_1\|_{L^8_{tx}}+\| u_2\|_{L^8_{tx}}+\| u_3\|_{L^8_{tx}})^2  \\
              &&  \|w_4\|_{L^\infty_t H^{3/2-}_x}  \|v_5 \|_{X^{1/2,1/2}}\\
              & \lesssim &  (\| w_1\|_{L^\infty_{t}H^{\frac{3}{2}-}}+\|u_2\|_{L^4_{t}H^s_4}+\|u_3\|_{L^4_{t}H^s_4})
              (\|w_1\|_{L^8_{tx}}+\| u_2\|_{L^8_{tx}}+\| u_3\|_{L^8_{tx}})^2  \\
              &&  \|w_4\|_{L^\infty_t H^{3/2-}_x}  \|v_5 \|_{X^{1/2,1/2}}
    \end{eqnarray*}
    where in the second  to the last step we used  Lemma \ref{leibrule1} . \\
         {\bf 2.}    $N_4 < 2^{-8} N_5 $.  Then $ N_{45} \sim N_5 $ so that we can drop the summation over $ N_{45}$ by replacing
          $P_{N_{45}} $ by $\tilde{P}_{N_5}$.  Note that in this region the frequency projections force $ N_5 \lesssim N_{123}$. \\
         {\bf 2.1.}    $N_4 \ge  2^{-8} N $.     By almost orthogonality it yields
             \begin{eqnarray*}
               I  & \lesssim  &\Bigl[ \sum_{N_{123}} \Bigl(\sum_{N_5\lesssim N_{123}}  \sum_{N_4 \lesssim N_5}   \sum_{2\le N\le N_{4}}
              N^s  \Bigl\|\partial_x^{\nu}  P_N P_{+}
         \Bigl( P_{N_{123}}  \partial_x^{-\nu}\Bigl( w_1 u_2 u_3 \Bigr) \\
         &&\quad\cdot\partial_x \tilde{P}_{N_{5}}P_{-}(P_{N_4} w_4 P_{N_5}v_5)\Bigr)
     \Bigr\|_{L^{4/3}_{tx}}\Bigr)^2 \Bigr]^{1/2}  \\
         & \lesssim  &\Bigl[ \sum_{N_{123}}\Bigl(N_{123}^{s}\|P_{N_{123}}(w_1 u_2 u_3) \|_{L^2_{tx}}\sum_{N_5\lesssim N_{123}}(\frac{N_5}{N_{123}})^{1/2}
        \|D_x^{1/2} P_{N_5}v_5\|_{L^4_{tx}} \Bigr)^2 \Bigr]^{1/2} \\
         & & \sum_{N_4,N}N^{0-}N_4^{\frac{1}{2}+}\|P_{N_4}w_4 \|_{L^\infty_{tx}} \\
         & \lesssim &  \|J^{s}_x(w_1 u_2 u_3) \|_{L^2_{tx}}  \|w_4\|_{L^\infty_t H^{3/2-}_x}  \|D_x^{1/2} v_5 \|_{L^4_{tx}}  \\
            & \lesssim &  (\| w_1\|_{L^\infty_{t}H^{\frac{3}{2}-}}+\|u_2\|_{L^4_{t}H^s_4}+\|u_3\|_{L^4_{t}H^s_4})
              (\|w_1\|_{L^8_{tx}}+\| u_2\|_{L^8_{tx}}+\| u_3\|_{L^8_{tx}})^2  \\
              &&  \|w_4\|_{L^\infty_t H^{3/2-}_x}  \|v_5 \|_{X^{1/2,1/2}}
    \end{eqnarray*}
      {\bf 2.2.}    $N_4 <  2^{-8} N $. \\
           {\bf 2.2.1.}
            $N_1\ge 2^{-8}  N_{123}$. Then we get
     \begin{eqnarray*}
       I & \lesssim &\sum_{N\ge 2} N^{s+\nu}\sum_{N_{123}\ge N} N_{123}^{-\nu} \sum_{N_1\gtrsim N_{123} }
       \sum_{N_2,N_3,N_4, N_5 \lesssim N_1}  \|P_{N_1}w_1 P_{N_2} u_2 P_{N_3} u_3\|_{L^2_{tx}} \\
       &&\quad\cdot N_5^{1/2} \|D_x^{1/2}v_5\|_{L^{4}_{tx}}  \|P_{N_4} w_4\|_{L^{\infty}_{tx}} \\
      & \lesssim &\sum_{N\ge 2} N^{0-}\sum_{N_{123}\ge N} N_{123}^{0-} \sum_{N_1}
      N_1^{s+\frac{1}{2}+} \|P_{N_1}w_1\|_{L^4_{tx} } ( \sum_{N_4}  N_4^{0-}\|w_4\|_{L^\infty_{tx}}) \\
       & & (\sum_{N_5 } N_5^{0-} \|D_x^{1/2} v_5\|_{L^{4}_{tx}}) \Bigr)  \prod_{i=2}^3
      \Bigl(  \sum_{N_i} N_i^{0-} \|P_{N_i} u_i\|_{L^{8}_{tx}}\Bigr) \\
       & \lesssim & \| w_1 \|_{L^\infty_t H^{s+1-}} \|w_4\|_{L^\infty_T H^{1/2}_x} \|v_5\|_{X^{1/2,1/2}} \prod_{i=2}^3  \|u\|_{L^\infty_t H^{1/2}_x}
      \end{eqnarray*}
     {\bf 2.2.2 } $N_1< 2^{-8}  N_{123}$. Then we have $ N_3 \sim N_{123}\sim N_{max} $. Since in this case it always hods $ 2^{-3} \le N_3/N_{123}\le 2^{3}$,
      by a slight abuse of notation we can drop the summation over $ N_{123} $ by replacing $ P_{N_{123}}$ by $ \tilde{P}_{N_3}$. \\
      {\bf 2.2.2.1} $N_1\ge 2^{-5} N_5 $.  Then  by almost orthogonality we get
     \begin{eqnarray*}
            I & \lesssim &\Bigl[ \sum_{N_3\ge 2}\Bigl( \sum_{2\le N\lesssim N_3 } N^{s+\nu} N_{3}^{-\nu}\sum_{N_1,N_2} \|P_{N_1}w_1 P_{N_2} u_2 P_{N_3} u_3\|_{L^\frac{8}{3}_{tx}}\\
            &&\quad\cdot \sum_{N_5\lesssim N_{1}} N_5 \|v_5\|_{L^{8}_{tx}}\sum_{N_4}\|P_{N_4} w_4 \|_{L^\infty_{tx}}
            \Bigr)^2 \Bigr]^{1/2}\\
     &  \lesssim &\Bigl[ \sum_{N_3\ge 2 }\Bigl( N_3^{s}\|P_{N_3}u_3\|_{L^4_{tx}}\sum_{N\lesssim N_3}(\frac{N}{N_3})^{s+\nu}\Bigr)^2\Bigr]^{1/2}
      \sum_{N_1}N_1^{1+}\|P_{N_1} w_1\|_{L^4_{tx}} \\
      & &  \sum_{N_3}
      N_2^{0-} \|D_x^{0+}P_{N_2} u_2\|_{L^{8}_{tx}} \sum_{N_4}
      N_4^{0-} \| D_x^{0+}P_{N_4} w_4\|_{L^{\infty}_{tx}} \sum_{N_5}
      N_5^{0-} \| D_x^{0+}P_{N_5} v_5\|_{L^{8}_{tx}} \\
       & \lesssim & \|w_1\|_{L^\infty_t H^{3/2-}}\| u_3\|_{\widetilde{L^4_{t}H^s_4}}  \| u_2 \|_{L^\infty_t H^{1/2}}
        \| w_4\|_{L^\infty_t H^{1/2}} \| v_5\|_{X^{1/2,1/2}}
      \end{eqnarray*}
             {\bf 2.2.2.2.}  $ N_1 < 2^{-5} N_5 $ and $ N_1 \ge 2^{-5} N $.  Then it holds
         \begin{eqnarray*}
     I &  \lesssim &  \sum_{N_2} \|P_{N_2}u_2\|_{L^4_{tx}}
      \sum_{N_1}N_1^{{1\over 2}+}\|P_{N_1} w_1\|_{L^\infty_{tx}} \sum_{N_3}
      N_3^{s-} \|P_{N_3} u_3\|_{L^{4}_{tx}}  \\
       & &  \sum_{N_4}   \|  P_{N_4} w_4\|_{L^{\infty}_{tx}}\sum_{N_5}
      N_5^{1/2-} \| P_{N_5} v_5\|_{L^{4}_{tx}} \\
       & \lesssim & \|w_1\|_{L^\infty_t H^{3/2-}_x}\| u_2\|_{L^\infty_{t}H^{1/2}}  \| u_3\|_{\widetilde{L^4_{t}H^s_4}} \|v_5\|_{X^{1/2,1/2}}
       \| w_4 \|_{L^\infty_t H^{3/2-}_x}\; .
      \end{eqnarray*}
        {\bf 2..2.2.3}  $ N_1 < 2^{-5} (N_5\wedge N)  $ and $ N_2 \ge 2^{-5} (N_5\wedge N) $.\\
   {\bf 2.2.2.3.1}   $ N_2 \ge 2^{-7} N $. Then either $ N_3\sim N $ and then $ N_3 \sim N\sim  N_2 $ which leads to
   \arraycolsep1pt
\begin{align*}
I  \lesssim  & \Bigl[\sum_{N_3\ge 2}\Bigl( \sum_{N_2\sim N_3} \sum_{
N_i\lesssim N_2 \atop i\in\{1,4,5\}}  \Bigl\|  \partial_x^{\nu}
{\tilde P}_{N_3} P_{+} \Bigl(\tilde{P}_{N_3}
\partial_x^{-\nu}\Bigl( P_{N_1}w_1 P_{N_2} u_2 P_{N_3} u_3\Bigr)\\
&\quad \cdot\partial_x \tilde{P}_{N_5}P_{-} (P_{N_4} w_4 P_{N_5}
v_5)\Bigr)
     \Bigr\|_{X^{s,-1/2}}\Bigr)^2\Bigr]^{1/2} \\
     \lesssim  & \Bigl[\sum_{N_3}  \Bigl( \sum_{N_2\sim N_3} \sum_{ N_i\lesssim N_2 \atop i\in\{1,4,5\}}  N_3^{s}
       \Bigl\|P_{N_1}w_1 P_{N_2} u_2 P_{N_3} u_3\Bigr\|_{L^2_{tx}} N_5  \|P_{N_4} w_4 P_{N_5} v_5\|_{L^4_{tx}}\Bigr)^2\Bigr]^{1/2} \\
                \lesssim &
       \Bigl[ \sum_{N_3}N_3^{2s} \|  P_{N_3}u_3 \|_{L^4_{tx}}^2 \Bigl( \sum_{N_1,N_4}\|P_{N_1}  w_1\|_{L^\infty_{tx}}
       \|P_{N_4}  w_4\|_{L^\infty_{tx}} \sum_{N_2\sim N_3}  \sum_{N_5\lesssim  N_3 } \Bigl( \frac{N_3}{N_2}\Bigr)^{1/2} \Bigl( \frac{N_5}{N_3}\Bigr)^{1/2} \\
       &  \quad   \|D_x^{1/2} P_{N_2}u_2\|_{L^4_{tx}}
      \|P_{N_5} D_x^{1/2} v_5 \|_{ L^4} \Bigr)^2 \Bigr]^{1/2} \\
      \lesssim & \|u_3 \|_{\widetilde{L^4_{t}H^s_4}} \| D_x^{1/2} u_2\|_{L^4_{tx}}  \| D_x^{1/2} v_5 \|_{L^4_{tx}} \|w_1 \|_{L^\infty_t H^{3/2-}}\|w_4 \|_{L^\infty_t H^{3/2-}}
\end{align*}
                   or $ N_3 \sim N_5 $ and then  we get
                       \begin{eqnarray*}
   I  & \lesssim  & \Bigl[\sum_{N}\Bigl( \sum_{N_2\gtrsim N}  \sum_{N_3\gtrsim N}  \sum_{ N_i,  i\in\{1,4\}}  N^{s+\nu} N_3^{1/2-\nu}    \| P_{N_1}w_1 \|_{L^\infty_{tx}}
   \|P_{N_2} u_2 \|_{L^4_{tx}} \\
    && \quad \|P_{N_3} u_3\|_{L^4_{tx}} \|P_{N_4} w_4\|_{L^\infty_{tx}} \| P_{N_3} D_x^{1/2} v_5\|_{L^4_{tx}}\Bigr)^2 \Bigr]^{1/2}\\
           & \lesssim & \|w_1\|_{L^\infty_t H^{1/2}_x}\|w_4\|_{L^\infty_{t}H^{1/2}_x}\sum_{N_3}N_3^{s} \|  P_{N_3}u_3 \|_{L^4_{tx}}
                    \|  P_{N_3}D_x^{1/2} v_5 \|_{L^4_{tx}} \\
& & \quad  \Bigl[\sum_{N}\Bigl( \sum_{N_2\gtrsim N}\Bigl( \frac{N}{N_2}\Bigr)^{1/2} N_2^{1/2}\|P_{N_2} u_2\|_{L^4_{tx}}\Bigr)^2 \Bigr]^{1/2}\\
                & \lesssim & \| u_3 \|_{\widetilde{L^4_{t}H^s_4}} \|u_2 \|_{\widetilde{L^4_{t}H^{1/2}_4}} \| D_x^{1/2} v_5 \|_{\tilde{L}^4_{tx}}
                \|w_1 \|_{L^\infty_t H^{3/2-}}\|w_4 \|_{L^\infty_t H^{3/2-}}
\end{eqnarray*}
where, in the last step, we used Cauchy-Schwarz in $ N_3 $ and that by discrete Young inequality
  $$
     \Bigl\|   \sum_{ k\in \Z}   (2^{k-k_2} )^{1/2} \chi_{\{k-k_2\le 5\}}   \|J_x^{1/2}P_{2^{k_2}}u_2 \|_{L^4} \Bigr\|_{l^2(\N)} \lesssim \| J_x^{1/2} u_2 \|_{ \tilde{L}^4_{tx}}\; .
     $$
      {\bf 2.2.2.3.2.}  $ N_2 < 2^{-7}N $ . Then $ N_2\ge 2^{-5} N_5 $
          since we must have   $ N_5\le 2^{-3}N_3$. This forces  $ N_3\sim N $ so that we get
           \begin{eqnarray*}
   I & \lesssim  & \Bigl[\sum_{N_3}  \Bigl( \sum_{N_2} \sum_{N_5\lesssim N_2} \sum_{ N_1,N_4}  N_3^{s}
       \Bigl\|P_{N_1}w_1 P_{N_2} u_2 P_{N_3} u_3\Bigr\|_{L^2_{tx}} N_5  \|P_{N_4} w_4 P_{N_5} v_5\|_{L^4_{tx}}\Bigr)^2\Bigr]^{1/2} \\
                & \lesssim & \|w_1\|_{L^\infty_t H^{1/2}_x}\|w_4\|_{L^\infty_{t}H^{1/2}_x}
       \Bigl[ \sum_{N_3}N_3^{2s} \|  P_{N_3}u_3 \|_{L^4_{tx}}^2\Bigr]^{1/2}   \\
&  &       \sum_{N_2}  \sum_{N_5\lesssim  N_2 } \Bigl( \frac{N_5}{N_2}\Bigr)^{1/2}
      \|J_x^{1/2} P_{N_2}u_2\|_{L^4_{tx}}
      \|P_{N_5} D_x^{1/2} v_5 \|_{ L^4}  \\
      & \lesssim & \|u_3 \|_{\widetilde{L^4_{t}H^s_4}}\| u_2 \|_{\widetilde{L^4_{t}H^{1/2}_4}}  \| D_x^{1/2} v_5 \|_{\tilde{L}^4_{tx}}
       \|w_1 \|_{L^\infty_t H^{3/2-}}\|w_4 \|_{L^\infty_t H^{3/2-}}
\end{eqnarray*}
where in the last step we used the discrete Young inequality.\\
          {\bf 2.2.2.4.}   $ (N_1\vee N_2) <  2^{-5} (N\wedge N_5)  $.
    Here it is worth noticing that we can assume that $(N \wedge N_5) \ge 2^4 $ and  the result follows  directly
    from the lemma  below and the proof of the proposition is completed.
    \end{proof}
    \begin{lemma} Under the same hypotheses on $ u_i $ as in Proposition, it holds
\begin{align*}
J:=\Bigr[& \sum_{N\ge 2^4} \Bigr( \sum_{(N_i)_{1\le i\le 5}\in
\Lambda_N} \Bigl\| \partial_x^{\nu}P_N
P_{+}\Bigl(\partial_x^{-\nu}\tilde{P}_{N_3} \Bigl( P_{N_1} w_1
P_{N_2}u_2 P_{N_3} u_3\Bigr)\\
&\cdot\partial_x \tilde{P}_{N_5}P_{-} (P_{N_4} w_4 P_{N_5}
v_5)\Bigr) \Bigr\|_{X^{s,-1/2+}} \Bigl)^2 \Bigr]^{1/2} \lesssim
\prod_{i=1}^3 \| u_i\|_{Z}
\end{align*}
where
\begin{align*}
\Lambda_N:=\Bigl\{&(N_1,N_2,N_3,N_4,N_5)\in (2^{\N} \cup \{0\})^5,
\, N_3\ge  2^{-3} N, \\ &2^4<N_5\le 4  N_3, \, (N_1\vee N_2 \vee
N_4)<2^{-5} (N \wedge N_5)\Bigl\}.
\end{align*}
    \end{lemma}
    \begin{proof}
   It is worth noticing that, thanks to  the frequency projections, $ N_3 \sim N_{max} $ and
      the resonance relation yields
      \begin{equation}
       |\sigma_{max}|\gtrsim |\xi \xi_5 | \ge 2^{-2} N N_5 \gtrsim (N\wedge N_5) N_3
       \end{equation}
        for all the contributions in $ J $.  First we can easily treat the contribution of the region $ \{(\tau,\xi),\, \langle \tau- \xi |\xi| \rangle \ge 2^{-2} N N_5 \} $.  Indeed, we then get
      \begin{eqnarray*}
   J   & \lesssim  &\sum_{N\ge 2^4 } \sum_{(N_1,N_2,N_3,N_4,N_5)\in \Lambda_N} ( N  N_5)^{-{1\over 2}+}  N_5^{1/2}  N^{s+\nu} N_3^{-s-\nu}  \\
    & &
   \|P_{N_3} D_x^{s} u_3 \|_{L^{4}_{tx}}  \| D_x^{1/2}P_{N_5} v_5 \|_{L^{4}_{tx}}
   \|P_{N_2}u_2 \|_{L^{\infty}_{tx}}    \prod_{i=1,4}  \|P_{N_i} w_i \|_{L^{\infty}_{tx}}
   \\
      & \lesssim & \sum_{N\ge 2^4}   N^{-\frac{1}{2}+}    \| u_3 \|_{L^4_{t}H^s_4}  \| D_x^{1/2} v_5 \|_{L^4_{tx}}
      \|u_2 \|_{L^\infty_t H^{1/2}_x}    \|w_1 \|_{L^\infty_t H^{1/2}_x}\|w_4 \|_{L^\infty_t H^{1/2}_x}
                \end{eqnarray*}
which is acceptable. Therefore in the sequel we can assume that $  \langle \tau- \xi |\xi| \rangle<
 2^{-2} N N_5 \} $.  Now, for any fixed couple $ (N,N_5) \in  (2^{\N})^2 $, we split any function $ z\in L^2_{tx}$ into two parts related to the value of $ \sigma $ by setting
   $$
  z= {\mathcal F}^{-1} \Bigl( \eta_{ 2^{-4}N N_5}(\tau-\xi |\xi|)  \widehat{z}\Bigr)
   +{\mathcal F}^{-1} \Bigl((1-\eta_{2^{-4}N N_5}(\tau-\xi|\xi|) ) \widehat{z}\Bigr) :=
   \tilde{z}+\tilde{\tilde{z}}\; .
   $$
  { \bf 1.} Contribution of  $ \tilde{\tilde{v}}_5 $.  We now control  the contribution of $ \tilde{\tilde{v}}_5 $ to $ J $  in the following way : either $ N\sim N_3\sim N_{max} $
  and we write
\begin{align*}
   J \lesssim  &\sum_{N\ge 2^4}  \sum_{N_3\sim N} \sum_{N_i \le 2^{5} N_3 \atop  i=1,3,4,5}
   N^{s}   ( N  N_5)^{-1/2}  N_5^{1/2}  N_3^{-s} \\
    & \| P_{N_5}  \tilde{\tilde{v}}_5  \|_{X^{1/2,1/2}}
   \|P_{N_3} D_x^{s} u_3 \|_{L^{4}_{tx}} \|P_{N_2} u_2 \|_{L^{\infty}_{tx}}
      \prod_{i=1,4} \|P_{N_i} w_i \|_{L^{\infty}_{tx}}
   \\
      \lesssim & (\sum_{N} N^{-\frac{1}{2}+} ) \| \tilde{\tilde{v}}_5 \|_{X^{1/2,1/2}}  \|   u_3 \|_{L^4_{t}H^s_4}
      \|u_2 \|_{L^{\infty}_{t}H^{1/2}_x }    \prod_{i=1,4} \|w_i \|_{L^{\infty}_{t}H^{1/2}_x }
\end{align*}
       or $ N_5\sim N_3\sim N_{max} $ and we write
        \begin{align*}
   J \lesssim  &\sum_{N\ge 2^4} \sum_{N_3} \sum_{N_5\sim N_3}  \sum_{N_i \le 2^{-5} N \atop  i=1,2,4}
   N^{s}   ( N  N_5)^{-1/2}  N_5^{1/2}  N_3^{-s} \\
    & \| P_{N_5}  \tilde{\tilde{v}}_5  \|_{X^{1/2,1/2}}
   \|P_{N_3} D_x^{s} u_3 \|_{L^{4}_{tx}}
     \|P_{N_2} u_2 \|_{L^{\infty}_{tx}} \prod_{i=1,4} \|P_{N_i} w_i \|_{L^{\infty}_{tx}}
   \\
      \lesssim & (\sum_{N} N^{-\frac{1}{2}+} ) \| v_5 \|_{X^{1/2,1/2}}  \|  u_3 \|_{\widetilde{L^4_{t}H^s_4}}
      \| u_2 \|_{L^{\infty}_{t}H^{1/2}_x}
        \prod_{i=1,4} \|w_i \|_{L^{\infty}_{t}H^{1/2}_x }
          \end{align*}
where we apply Cauchy-Schwarz in $ N_3\sim N_5 $ in the last step.\\
  { \bf 2.} Contribution of  $ \tilde{v}_5 $. \\
   { \bf 2.1} Contribution of $ \tilde{\tilde{w}}_1 $.
    We easily get
 \begin{align*}
   J  \lesssim  &
   \sum_{N}  \sum_{(N_1,N_2,N_3,N_4,N_5)\in \Lambda_N} ( (N \wedge N_5)N_3)^{-1} N_1^{\frac{1}{2}+} N_5^{1-s}  N^{s} \\
    & \| P_{N_1} \tilde{\tilde{w}}_1 \|_{X^{-1/2-,1}}    \|P_{N_3} u_3 \|_{L^{\infty}_{tx}}
   \|P_{N_5} D_x^{s} \tilde{v}_5 \|_{L^{4}_{tx}}
  \|P_{N_2} u_2 \|_{L^{\infty}_{tx}}  \|P_{N_4} w_4 \|_{L^{\infty}_{tx}}
   \\
      \lesssim &    \sum_{N} N^{-\frac{1}{2}+} \|w_1\|_{X^{-1/2-,1}}
        \|   u_3\|_{L^\infty_{t}H^{1\over 2}_x}\|   v_5 \|_{X^{s,1/2}}
        \|u_2 \|_{L^{\infty}_{t}H^{1/2}_x }    \|w_4 \|_{L^{\infty}_{t}H^{1/2}_x }
          \end{align*}
which is acceptable.\\
{ \bf 2.2} Contribution of $\tilde{w}_1 $.
   To treat  this contribution  we will extensively use the following lemma which is a direct application of the Marcinkiewicz multiplier theorem.
      \begin{lemma} For any $ p\in ]1,+\infty[ $
   there exists $ C_p> 0$ such that for all  $ N \ge 1 $ and  all $ L\ge N^2  $,
   \begin{equation}\label{mikhlin}
  \Bigl\| {\mathcal F}^{-1}_{tx}\Bigl( \phi_{N} (\xi)  \eta_{ L}(\tau\mp\xi^2) f(\tau,\xi)
  \Bigr) \Bigr\|_{L^p_{tx}} \le C_p \, \| f\|_{L^p_{tx}} \, , \quad \forall f\in L^p(\R^2) \; .
  \end{equation}
   \end{lemma}
   \begin{proof} By  Marcinkiewicz  multiplier theorem (see for instance (\cite{G},Corollary 5.2.5 page 361)), it suffices to check that
   $$
 \Bigl|  \partial^{(\alpha_1,\alpha_2)}_{\tau,\xi} \Bigl(  \phi_{N} (\xi)  \eta_{ L}(\tau\mp \xi^2) \Bigr)\Bigr| \lesssim |\xi|^{\alpha_1}  |\tau|^{\alpha_2}
   \mbox{ for } |\alpha|\le  2.
    $$
    But this follows directly from the fact that for $ N^2\le L $ ,
    $$
    \frac{d}{d\xi} \Bigl(  \phi_{N} (\xi)  \eta_{ L}(\tau\mp \xi^2)\Bigr)=O(N^{-1}) \mbox{ and }\frac{d}{d\tau}
     \Bigl(  \phi_{N} (\xi)  \eta_{ L}(\tau\mp \xi^2)\Bigr) =O(L^{-1})\; .
    $$
   \end{proof}
It is worth noticing that on $ \Lambda_N $, with $ N\ge 2^4 $,  it holds $ N_i^2 \le 2^{-2} N N_5 $ for $i\in\{1,2,4\}$.
 Hence, in view of  \eqref{mikhlin}, for any $ 1<p<\infty $, setting   $(z_1,z_2,z_4)= (w_1,u_2,w_4) $ it holds
 $$
  \| P_{N_i} P_{\mp} \tilde{z}  \|_{L^p_{tx}} \le\, C_p \,  \| P_{N_i}  P_{\mp} z\|_{L^p_{tx}}\; .
 $$
 and thus by the continuity of the Hilbert transform in $ L^p $, $1<p<\infty $,
 \begin{equation} \label{tata}
  \| P_{N_i}  \tilde{z}  \|_{L^p_{tx}} \le\, \tilde{C}_p \,  \|    P_{N_i} z\|_{L^p_{tx}}\; .
  \end{equation}
   We separate the contribution of
 $ \tilde{w}_1 $ in different sub-contributions. \\
   {\bf 2.2.1}  Contribution of $ \tilde{\tilde{u}}_2 $. Then we write
   \begin{align*}
   J  \lesssim  &
   \sum_{N}  \sum_{(N_1,N_2,N_3,N_4,N_5)\in \Lambda_N} ( (N\wedge   N_5)N_3)^{-1} N_2^{\frac{1}{2}+} N_5  N^{s} N_3^{-s+1/6}\\
    & \| P_{N_2} \tilde{\tilde{u}}_2 \|_{X^{-1/2-,1}}    \|P_{N_3} D_x^{s-1/6}u_3 \|_{L^{6}_{tx}}
   \|P_{N_5} \tilde{v}_5 \|_{L^{24}_{tx}}   \|P_{N_1} \tilde{w_1} \|_{L^{24}_{tx}}
  \|P_{N_2} u_2 \|_{L^{\infty}_{tx}}  \|P_{N_4} w_4 \|_{L^{\infty}_{tx}}
   \\
      \lesssim &    \sum_{N} N^{-\frac{1}{6}} \|w_1\|_{L^{24}_{tx}}
        \|  D_x^{s} u_3\|_{L^6_t L^3_x}\|   v_5 \|_{X^{1/2,1/2}}
        \|u_2 \|_{X^{-1/2-,1}}    \|w_4 \|_{L^{\infty}_{t}H^{1/2}_x } \\
           \lesssim &    \|w_1\|_{L^{\infty}_t H^{1/2}_x}
        \|   u_3\|_{L^\infty_t H^{s}\cap L^4_t H^{s}_4}\|   v_5 \|_{X^{1/2,1/2}}
        \|u_2 \|_{X^{-1/2-,1}}     \|w_4 \|_{L^{\infty}_{t}H^{1/2}_x }
          \end{align*}
where we used Sobolev inequalities and \eqref{tata} in the last to the last step.\\
{\bf 2.2.2} Contribution of $ \tilde{u}_2 $.\\
{\bf 2.2.2.1} Contribution of $ \tilde{\tilde{w}}_4 $. This subcontribution can be estimated in the same way by
 \begin{align*}
   J  \lesssim  &
   \sum_{N}  \sum_{(N_1,N_2,N_3,N_4,N_5)\in \Lambda_N} (  (N\wedge   N_5)  N_3)^{-1} N_4^{\frac{1}{2}+} N_5   N^{s} N_3^{-s+1/6}\\
    & \| P_{N_2} \tilde{\tilde{w}}_4 \|_{X^{-1/2-,1}}    \|P_{N_3} D_x^{s-1/6}u_3 \|_{L^{6}_{tx}}
   \|P_{N_5}  \tilde{v}_5 \|_{L^{36}_{tx}}   \|P_{N_1} \tilde{w_1} \|_{L^{36}_{tx}}
   \|P_{N_2} \tilde{u}_2 \|_{L^{36}_{tx}} \\
             \lesssim &    \|w_1\|_{L^{\infty}_t H^{1/2}_x}
         \|   u_3\|_{L^\infty_t H^{s}\cap L^4_t H^{s}_4}\|   v_5 \|_{X^{1/2,1/2}}
        \|u_2 \|_{L^{\infty}_{t}H^{1/2}_x }    \|w_4 \|_{X^{-1/2-,1}}
          \end{align*}
{\bf 2.2.2.2}  Contribution of $ \tilde{w}_4 $.
 Since $ \max(|\sigma_i|) < 2^{-2} N N_5 $, we only have to consider
  $ \tilde{\tilde{u}}_3 $ in this contribution. Either $ N \sim N_3 $ and
 then
  \begin{align*}
   J^2 \lesssim  &\sum_{N_3} \Bigl( \sum_{(N_1,N_2,N_3,N_4,N_5)\in \Lambda_{N_3}}
   N_3^{s}   ( N_3 N_5)^{-1}  N_5^{2/3}  N_3^{1-s} \\
    & \| P_{N_3}  \tilde{\tilde{u}}_3  \|_{X^{s-1,1}}
   \|P_{N_5} D_x^{1/3}  \tilde{v}_5 \|_{L^{6}_{tx}}
      \|P_{N_2} \tilde{ u}_2 \|_{L^{36}_{tx}} \prod_{i=1,4} \|P_{N_i} \tilde{ w}_i \|_{L^{36}_{tx}}\Bigr)^2
   \\
      \lesssim &  \Bigl( \sum_{N_3} \|P_{N_3}  u_3 \|_{X^{s-1,1}} \Bigr)^2
     \|   v_5 \|_{X^{1/2,1/2}}^2   \|u_2 \|_{L^{\infty}_{t}H^{1/2}_x }^2  \prod_{i=1,4} \|w_i \|_{L^{\infty}_{t}H^{1/2}_x }^2
                   \end{align*}
or $ N_5\sim N_3 $.  In this last case we first notice that  $X^{0,3/8}  \hookrightarrow L^4_{tx} $ and
 that for any fixed $ 2<p<\infty $,   $  X^{0,1/2-}_T \hookrightarrow L^p_T L^2_{x} $. Therefore by interpolation, Sobolev inequalities and
  duality we infer that $ L^{6\over 5}_{Tx}  \hookrightarrow X^{-\frac{1}{6}-,-1/2+}_T  $. We thus get
  \begin{align*}
   J^2 \lesssim  &\sum_{N} \Bigl(  \sum_{N_3}  \sum_{(N\vee N_1\vee N_2\vee N_4)\lesssim N_3}
   N^{s} N^{\frac{1}{6}+} ( N N_3)^{-1}  N_3^{1/2}
  N_3^{1-s} \\
    & \| P_{N_3}  \tilde{\tilde{u}}_3  \|_{X^{s-1,1}}
   \|P_{N_3} D_x^{1/2} \tilde{v}_5 \|_{L^{4}_{tx}}
        \|P_{N_2}   \tilde{ u}_2 \|_{L^{36}_{tx}} \prod_{i=1,4}  \|P_{N_i}\tilde{w}_i \|_{L^{36}_{tx}}\Bigr)^2
   \\
      \lesssim & \sum_{N} N^{-\frac{1}{6}+}\Bigl(  \sum_{N_3} \|P_{N_3}  u_3 \|_{X^{s-1,1}}
       \|  P_{N_3}D_x^{1/2}\tilde{v}_5 \|_{L^4_{tx}} \Bigr)^2
       \|u_2 \|_{L^{\infty}_{t}H^{1/2}_x }^2  \prod_{i=1,4} \|w_i \|_{L^{\infty}_{t}H^{1/2}_x }^2\\
                 \lesssim &\|  u_3 \|_{X^{s-1,1}}^2
           \|   v_5 \|_{X^{1/2,1/2}}^2   \|u_2 \|_{L^{\infty}_{t}H^{1/2}_x }^2
        \prod_{i=1,4} \|w_i \|_{L^{\infty}_{t}H^{1/2}_x }^2
          \end{align*}
          where we apply Cauchy-Scwharz in $ N_3 $ in the last step.
                  \end{proof}
 Finally, this last proposition together with Lemmas \ref{lemma43} enables to treat   the term containing $B(u,u) $ in \eqref{est}.
     \begin{proposition} Let $ 1/2\le s\le 1 $,  $ w_1 \in X^{-1/2-,1}\cap L^\infty_t H^{s+1-}_x  $,  and $
u_i\in X^{s-1,1}\cap  \widetilde{L^4_t H^{s}_4} \cap L^\infty_t H^{s}$, $ i=2,3,4$.
 with compact support in time such that $ u_2 $ and $ u_3 $ are real-valued. Then it holds
\begin{align}
\Bigl\| w_1 u_2   B(u_3,u_4) \Bigr\|_{X^{s,-{1\over 2}+}} \lesssim&
\| w_1 \|_{X^{-1/2-,1}}  \|u_2\|_{L^\infty_t H^{1\over 2}} \|  u_3 \|_{ {L}^\infty_{t}H^{1\over 2}}  \|  u_4\|_{L^4_{t}H^s_4}  \nonumber  \\
& + \| w_1 \|_{L^\infty_t H^{{3\over 2} -}_x}  \sum_{(2\le i\neq
j\neq q\le 4)}
 \| u_i \|_{X^{-{1\over 2}-,1} \cap\widetilde{L^4_{t}H^{1\over 2}_4}\cap L^\infty_t H^{1\over 2}}
\nonumber\\
&\quad\cdot\|u_j\|_{X^{-{1\over 2},1}
\cap\widetilde{L^4_{t}H^{1\over 2}_4}\cap L^\infty_t H^{1\over 2}}
  \|u_q\|_{X^{s-1,1} \cap\widetilde{L^4_{t}H^{s}_4}\cap L^\infty_t H^{s}} \; .
\end{align}
\end{proposition}
\begin{proof}
Recall that $ B(u,v) = -i\partial_x^{-1} ( P_+ u_x P_+ v_x) + i\partial_x^{-1} ( P_- u_x P_- v_x) $. By symmetry it thus suffices to estimate
\begin{align*}
I:=&\Bigl\|  w_1 u_2 \partial_x^{-1} (P_+ \partial_x u_3 P_+ \partial_x u_4)\Bigr\|_{X^{s,-{1\over 2}+}}\\
  =& \Bigl\| \sum_{N\ge 1 ,N_1,N_2,N_3,N_4, N_{34}\ge (N_3\vee
N_4)/2}  P_N \Bigl( P_{N_1}w_1 P_{N_2} u_2
\\
&\qquad \cdot P_{N_{34}}\partial_x^{-1} (P_+ \partial_x P_{N_3} u_3
P_+
    \partial_x P_{N_4}u_4) \Bigr)   \Bigr\|_{X^{s,-{1\over 2}+}} \; .
\end{align*}
     By symmetry we can assume that $ N_3\le N_4 $ and thus we must have $ N_{34} \sim N_4 $. We can thus drop the summation over $ N_{34}$ and replace $ P_{N_{34}}$ by $ \tilde{P}_{N_4} $.

     By the triangle inequality we can separate this sum in different sums on disjoint subset of $ (2^{\N})^5 $. \hspace*{2mm} \\
        {\bf 1.}    $N_1 \ge 2^{-8} N $. Then we  have
\begin{align*}
   I \lesssim  &  \sum_{N} \sum_{N_1\ge 2^{-8}N}   \sum_{N_2,N_3,N_4} N^{s}\Bigl\|  P_N \Bigl( P_{N_1}w_1 P_{N_2} u_2\tilde{P}_{N_{4}}\partial_x^{-1} (P_+ \partial_x  P_{N_3} u_3 P_+
    \partial_x P_{N_4}u_4)  \Bigr)
     \Bigr\|_{L^{4/3}_{tx}}  \\
     \lesssim &  \sum_{N \le 2^{8}N_1}   \sum_{N_2,N_3,N_4}  \|P_{N_1} w_1 \|_{L^\infty_t H^{s+}_8}Ê\|P_{N_2} u_2 \|_{L^\infty_t L^8_x}?   \prod_{i=3}^4   \|P_{N_i} u_i \|_{L^4_t H^{{1\over 2 }-}_4}?    \\
           \lesssim & \| w_1 \|_{L^\infty_t H^{{3\over 2}-}_x}   \|  u_2 \|_{L^\infty_t H^{1\over 2}_x}  \prod_{i=3}^4   \|P_{N_i} u_i \|_{L^4_t H^{{1\over 2 }}_4}? \; .
\end{align*}
     {\bf 2.}    $N_1< 2^{-8} N $ and $ N_1 \ge 2^{-5} N_3 $. Then either $  N_4\gtrsim N\vee N_2 $ and it holds
       \begin{align*}
   I  \lesssim  &  \Bigl[\sum_{N} \Bigl(\sum_{N_4\gtrsim N} \sum_{N_1< 2^{-8}N} \sum_{N_3\le 2^5 N_1}   \sum_{N_2} N^{s}\Bigl\|  P_N \Bigl( P_{N_1}w_1 P_{N_2} u_2\\
   &\qquad\cdot\tilde{P}_{N_{4}}\partial_x^{-1} (P_+ \partial_x  P_{N_3} u_3 P_+
    \partial_x P_{N_4}u_4)  \Bigr)
     \Bigr\|_{L^{4/3}_{tx}} \Bigr)^2  \Bigr]^{1\over 2}\\
     \lesssim &  \Bigl(\sum_{N} \Bigl(  \sum_{N_4\gtrsim N }  ({N\over N_4})^{2s} \|J_x^s u_4\|_{L^4_{tx}}^2 \Bigr)^{1\over 2}  \sum_{N_1,N_2,N_3}     \|P_{N_1} w_1 \|_{L^\infty_t H^{1}_4}Ê\|P_{N_2} u_2 \|_{L^\infty_t L^8_x}ÊÊ\|P_{N_2} u_3  \|_{L^\infty_t L^8_x}?  \\
              \lesssim & \| w_1 \|_{L^\infty_t H^{{3\over 2} -}_x}   \|  u_2 \|_{L^\infty_t H^{1\over 2}_x}
               \|  u_3 \|_{L^\infty_t H^{1\over 2}_x}  \| u_4 \|_{\widetilde{L^4_{t}H^s_4}}\; .
    \end{align*}
     or $ N_2 \gtrsim N\vee N_4 $ and it holds
      \begin{align*}
   I  \lesssim  &  \Bigl[\sum_{N} \Bigl(\sum_{N_2 \gtrsim N} \sum_{N_1< 2^{-8}N} \sum_{N_3\le 2^5 N_1}   \sum_{N_4} N^{s}\Bigl\|  P_N \Bigl( P_{N_1}w_1 P_{N_2} u_2\\
   &\quad \cdot\tilde{P}_{N_{4}}\partial_x^{-1} (P_+ \partial_x  P_{N_3} u_3 P_+
    \partial_x P_{N_4}u_4)  \Bigr)
     \Bigr\|_{L^{4/3}_{tx}} \Bigr)^2  \Bigr]^{1\over 2}\\
     \lesssim &\Bigl(\sum_{N} \Bigl(  \sum_{N_2\gtrsim N }  ({N\over N_2})^{2s} \|J_x^s u_2\|_{L^4_{tx}}^2 \Bigr)^{1\over 2}
      \sum_{N,N_3,N_4}     \|P_{N_1} w_1 \|_{L^\infty_t H^{1}_4}?     \prod_{i=3}^4 \|P_{N_i} v_i \|_{L^\infty_t L^8_x}? \\
              \lesssim & \| w_1 \|_{L^\infty_t H^{{3\over 2} -}_x}   \|  u_2 \|_{\widetilde{L^4_{t}H^s_4}}
              \prod_{i=3}^4    \|  u_i \|_{L^\infty_t H^{1\over 2}_x}  \; .
    \end{align*}
    {\bf 3.} $N_1< (2^{-8} N \wedge  2^{-5} N_3) $. \\
       {\bf 3.1}    $N_2 \ge 2^{-8} N $. Then we  have
          \begin{align*}
   I  \lesssim  &\Bigl( \sum_{N} \Bigl[  \sum_{N_2\ge 2^{-8}N}   \sum_{N_3,N_4} \sum_{N_1} N^{s}\Bigl\|  P_N \Bigl( P_{N_1}w_1
    P_{N_2}u_2 \\
    &\quad\cdot\tilde{P}_{N_{3}}\partial_x^{-1} (P_+ \partial_x  P_{N_3} u_3 P_+
    \partial_x P_{N_4}u_4)  \Bigr)
     \Bigr\|_{L^{4/3}_{tx}}  \Bigr]^2 \Bigl)^{1/2} \\
     \lesssim &\|w_1\|_{L^\infty_t H^1_x}  \Bigl( \sum_{N} \Bigl[  \sum_{N_2\ge 2^{-8}N} (\frac{N}{N_2})^{s} N_2^s \|P_{N_2}u_2\|_{L^4_{tx}} \Bigr]^2 \Bigr)^{1/2}
     \\
     &\cdot\sum_{N_4} N_4^{-1}\|\tilde{P}_{N_{4}}(P_+ \partial_x  P_{N_3} u_3 P_+
    \partial_x P_{N_4}u_4)\|_{L^2_{tx}} \\
    \lesssim & \|w_1\|_{L^\infty_t H^1_x}   \|  u_2 \|_{\widetilde{L^4_{t}H^s_4}}   \sum_{N_4}  \|P_{N_4} D_x^{1/2} u_4 \|_{L^4_{tx}} \sum_{N_3\le N_4} (\frac{N_3}{N_4})^{1/2}
     \|P_{N_3} D_x^{1/2} u_3 \|_{L^4_{tx}}  \\
      \lesssim & \|w_1\|_{L^\infty_t H^1_x} \| u_2 \|_{\widetilde{L^4_{t}H^s_4}}   \| D_x^{1/2} u_3 \|_{ \tilde{L}^4_{tx}}  \| D_x^{1/2} u_4\|_{ \tilde{L}^4_{tx}} \; .
    \end{align*}
    where we use two times the discret Young inequality. \\
     {\bf 3.2}    $N_2< 2^{-8} N $ and $ N_2 \ge 2^{-5} N_3 $.  Then we must have $  N\sim N_4$ and thus
        \begin{align*}
   I  \lesssim  &\Bigl( \sum_{N_4} \Bigl[  \sum_{N_1,N_3}\sum_{N_2\ge 2^{-5}N_3} N_2^{-1/2}N_3^{1/2}
   \|P_{N_1} w_1 \|_{L^\infty_{tx}} \|P_{N_2}J_x^{1/2}u_2 \|_{L^4_{tx}} \\
   &\quad\cdot\|P_{N_3}D_x^{1/2} u_3\|_{L^4_{tx}} \|P_{N_4} D_x^{s}u_4 \|_{L^4_{tx}}\Bigr]^2 \Bigr)^{1/2} \\
    \lesssim  &  \|w_1\|_{L^\infty_t H^1_x}    \Bigl( \sum_{N_4} \|D^s_x P_{N_4} u_4 \|_{L^4_{tx}}^2 \Bigr)^{1/2}
     \\
     &\cdot\sum_{N_2}  \|P_{N_2} J_x^{1/2} u_2 \|_{L^4_{tx}} \sum_{N_3\le 2^{5}N_2} (\frac{N_3}{N_2})^{1/2}
     \|P_{N_3} D_x^{1/2} u_3 \|_{L^4_{tx}}  \\
      \lesssim & \|w_1\|_{L^\infty_t H^1_x}   \|  u_2 \|_{\widetilde{L^4_{t}H^{1/2}_4}}   \|u_3 \|_{\widetilde{L^4_{t}H^{1/2}_4}}  \| u_4\|_{\widetilde{L^4_{t}H^s_4}} \; .
    \end{align*}
     {\bf 3.3} $N_2< (2^{-8} N \wedge  2^{-5} N_3) $. Then $ N\sim N_4 $ and
      the   resonance relation yields
      \begin{equation}
       |\sigma_{max}|\gtrsim |\xi_3\xi_4 | \ge 2^{-2} N_3 N_4 \; .
       \end{equation}
First we can easily treat the contribution of the region $ \{(\tau,\xi),\, \langle \tau- \xi |\xi| \rangle \ge 2^{-2} N_3 N_4 \} $.  Indeed, we then get
     \begin{align*}
   I  \lesssim  & \sum_{N_4}  \sum_{N_3\lesssim N_4} \sum_{N_1\vee N_2 \lesssim N_3}N_3^{1/2}(N_3 N_4)^{-1/2+}
    \|w_1\|_{L^\infty_t H^1_x}  \|P_{N_2}  u_2 \|_{L^\infty_{tx}} \\
    &\quad \cdot\|P_{N_3}D_x^{1/2} u_3\|_{L^4_{tx}}  \|P_{N_4} J_x^{s}u_4 \|_{L^4_{tx}} \\
   \lesssim &  \|w_1\|_{L^\infty_t H^1_x} \| D_x^{1/2} u_3 \|_{ {L}^4_{tx}}  \|  u_4\|_{{L^4_{t}H^s_4}}
   \sum_{N_2} N_2^{-1/2+}\|P_{N_2}  u_2 \|_{L^\infty_{tx}} \\
    \lesssim  &   \|w_1\|_{L^\infty_t H^1_x} \| u_2\|_{ {L}^\infty_{t}H^{1/2}_x}   \|  u_3 \|_{ L^4_t H^{1\over 2}_4}  \| u_4\|_{{L^4_{t}H^s_4}}  \; .
    \end{align*}
which is acceptable. Therefore in the sequel we can assume that $  \langle \tau- \xi |\xi| \rangle<
 2^{-2} N_3 N_4 \} $.  We now split $ v_1 $, $u_2 $  and $ u_3 $   into two parts related to the value of $ \sigma_i $ by setting
   $$
  z= {\mathcal F}^{-1} \Bigl( \eta_{ 2^{-4}N_3 N_4}(\tau-\xi |\xi|)  \widehat{z}\Bigr)
   +{\mathcal F}^{-1} \Bigl((1-\eta_{2^{-4}N_3 N_4}(\tau-\xi|\xi|) ) \widehat{z}\Bigr) :=
   \tilde{z} +\tilde{\tilde{z}} \; .
   $$
   It is worth noticing that in this region $ N_i^2<< N_3 N_4 $ for $ i=1,2,3$. Therefore Lemma \ref{mikhlin} holds for   $\tilde{w}_1,\, \tilde{u}_2 $
    and    $  \tilde{u}_3 $.\\
     { \bf 3.3.1.} Contribution of  $ \tilde{\tilde{w}}_1 $.  We first control  the contribution of $ \tilde{\tilde{w}}_1 $ to $ I $  in the following way :
  \begin{align*}
   I \lesssim  &   \sum_{N_4} \sum_{N_3\lesssim N_4} \sum_{N_1\vee N_2 \lesssim N_3} N_1^{1/2+}N_3 (N_3 N_4)^{-1} \|P_{N_1}  \tilde{\tilde{w}}_1 \|_{X^{-1/2-,1}}  \\
   &\quad\cdot\|P_{N_2}  u_2 \|_{L^\infty_{tx}}  \|P_{N_3} u_3\|_{L^\infty_{tx}} \|P_{N_4} J_x^{s}u_4 \|_{L^4_{tx}} \\
    \lesssim  &  \| w_1 \|_{X^{-1/2-,1}}  \|u_2\|_{L^\infty_t H^{1\over 2}} \|  u_3 \|_{ {L}^\infty_{t}H^{1\over 2}}  \|  u_4\|_{L^4_{t}H^s_4}
 \; .
    \end{align*}
    \noindent
   {\bf 3.3.2.}  Contribution of  $ \tilde{w}_1 $.\\
  { \bf 3.3.2.1} Contribution of  $ \tilde{\tilde{u}}_2 $. In the same way, using Sobolev inequality, we have
  \begin{align*}
   I \lesssim  &   \sum_{N_4} \sum_{N_3\lesssim N_4} \sum_{N_1\vee N_2 \lesssim N_3} N_1^{1/2+}N_3 (N_3 N_4)^{-1}N_4^{1/6}
  \|P_{N_1}  \tilde{w}_1 \|_{L^{12}_{tx}}   \|P_{N_2}  \tilde{\tilde{u}}_2 \|_{X^{-1/2-,1}} \\
  & \quad \quad   \|P_{N_3} u_3\|_{L^{\infty}_{tx}} \|P_{N_4} J_x^{s-1/6}u_4 \|_{L^6_{tx}} \\
    \lesssim  &  \|w_1\|_{L^\infty_t H^{1\over 2}} \|u_2\|_{X^{-1/2-,1}}  \|  u_3 \|_{L^\infty_t H^{1\over 2}}  \|  u_4\|_{L^\infty_t H^s \cap L^4_{t}H^s_4}
 \; .
    \end{align*}
    {\bf 3.3.2.2} Contribution of $\tilde{u}_2 $.  \\
          {\bf 3.2.2.2.1} Contribution of $\tilde{\tilde{u}}_3 $.
 \begin{align*}
   I \lesssim  &   \sum_{N_4} \sum_{N_3\lesssim N_4} \sum_{N_1\vee N_2 \lesssim N_3}  N_3^{3/2}(N_3 N_4)^{-1}N_4^{1/6}
    \|P_{N_1}  \tilde{w}_1 \|_{L^{24}_{tx}}   \\
    &\quad\cdot\|P_{N_1} \tilde{u}_2 \|_{L^{24}_{tx}}
    \|P_{N_3} \tilde{u}_3\|_{X^{-1/2,1}} \|P_{N_4} J_x^{s-1/6}u_4 \|_{L^6_{tx}} \\
    \lesssim  &  \|w_1\|_{L^\infty_t H^{1\over 2}}  \|u_2\|_{L^\infty_t H^{1\over 2}}    \|  u_3 \|_{X^{-1/2,1}}
  \|  u_4\|_{L^\infty_t H^s \cap L^4_{t}H^s_4}   \; .
    \end{align*}
      {\bf 3.2.2.2.2} Contribution of $\tilde{u}_3 $. Since $ \max(|\sigma_i|) \ge 2^{-2} N_3 N_4 $, it remains to
 treat the subcontribution of  $ \tilde{\tilde{u}}_4 $. We  easily obtain
 \begin{align*}
   I^2 \lesssim  & \sum_{N_4} \Bigl(  \sum_{N_3\lesssim N_4} \sum_{N_1\vee N_2\lesssim N_3}N_3^{2/3}N_4 (N_3 N_4)^{-1}
 \|P_{N_1}  \tilde{w}_1 \|_{L^{24}_{tx}}   \\
 &\quad\cdot\|P_{N_1} \tilde{u}_2 \|_{L^{24}_{tx}} \|P_{N_3} D_x^{1/3} P_+ \tilde{u}_3\|_{L^6_{tx}}  \|P_{N_4}  \tilde{\tilde{u}}_4\|_{X^{s-1,1}}\Bigr)^2  \\
       \lesssim  & \Bigl(   \|w_1\|_{L^\infty_t H^{1\over 2}}  \|u_2\|_{L^\infty_t H^{1\over 2}}
  \|  u_3\|_{L^\infty_t H^{1\over 2} \cap L^4_{t}H^{1\over 2}_4}   \|  u_4 \|_{X^{s-1,1}} \Bigr)^2 \; .
    \end{align*}
Therefore, we complete the proof.
\end{proof}

\subsection*{Acknowledgment}
Z. Guo is supported in part by NNSF of China (No. 11001003, No.
11271023).

\end{document}